\documentclass{amsart}
\usepackage{amsmath,amsthm,amsfonts,amssymb,amscd}
\usepackage{framed}
\usepackage{xy}[all]
\usepackage{nicefrac}
\usepackage{color,graphicx}
\usepackage[x11names, rgb]{xcolor}
\usepackage{tikz}
\usepackage{tikz-cd}
\usepackage{hyperref}
\usetikzlibrary{cd}
\usepackage{hyperref}
\usepackage{pdfsync}
\usepackage{tikz-cd}
\usepackage{enumitem}
\usepackage[capitalize,nameinlink]{cleveref}
\Crefname{conjecture}{Conjecture}{Conjectures}
\usepackage[utf8]{inputenc}
\usepackage{tableau}
\usepackage[position=bottom]{subfig}
\usepackage{youngtab}
\usepackage{amsthm}
\newtheorem{definition}{Definition}
\newtheorem{thm}{Theorem}[section]

\newtheorem{prop}[thm]{Proposition}
\newtheorem{lemma}[thm]{Lemma}
\newtheorem{Corollary}[thm]{Corollary}
\newtheorem{example}[thm]{Example}

\newtheorem{remark}[thm]{Remark}

\newcommand{\C}{{\mathbb C}}

\DeclareMathOperator{\K}{K}

\DeclareMathOperator{\ch}{ch}

\begin{document}
\title[I-functions on cotangent bundles]{Lifting $I$-functions from the Grassmannians to their cotangent bundles}

\author{Kamyar Amini}
\address{
Department of Mathematics, 
225 Stanger Street, McBryde Hall,
Virginia Tech University, 
Blacksburg, VA 24061
USA
}
\email{kamini@vt.edu}

\begin{abstract} 
We relate two fundamental enumerative functions, namely the $I$-functions in the quantum $K$-ring of $G(r,n)$ and of its cotangent bundle, by defining a $K$-theoretic operator on classes, called balancing. This operator lifts the $I$-function of $G(r,n)$ to that of $T^*G(r,n)$, providing an explicit geometric interpretation. We also define an operator acting on difference operators and show that, for certain $K$-theoretic classes and the corresponding difference operators that annihilate them—including the $I$-functions of projective spaces $\mathbb{P}^n$—the balancing operation on difference operators and on classes is compatible. Moreover, for general $G(r,n)$, we recover the Bethe-Ansatz equations for $T^*G(r,n)$ via a procedure inspired by both balancing and the abelian/non-abelian correspondence.
\end{abstract}

\maketitle 

\setcounter{tocdepth}{1}
\tableofcontents


\section{Introduction}

Let $X$ be a nonsingular projective variety equipped with an action of the torus $T = (\mathbb{C}^*)^n$. In Gromov-Witten theory, the $K$-theoretic $J$-function of $X$ is defined as:
\[
J(q, Q) = 1+\frac{1}{1-q}\sum_{\beta \in \mathrm{Eff}(X)}Q^{\beta}ev_*\Big(\frac{1}{1-q\mathcal{L}}\Big),
\]
where $ev_*:K_T(\overline{M}_{0,1}(X,\beta)) \to K_T(X)$ is the $K$-theoretic push forward induced by the evaluation map $ev:\overline{M}_{0,1}(X,\beta) \to X$, and $\mathcal{L}$ denotes the universal cotangent line bundle at the marked point. This function serves as a generating function for 1-point Gromov-Witten invariants.

When $X$ is noncompact, one requires a replacement framework in order to define invariants and quantum $K$-ring. Nakajima quiver varieties, including cotangent bundles of Grassmannians, provide important examples of such noncompact spaces. The new framework is the moduli space of quasimaps, introduced by Ciocan-Fontanine, Kim, and Maulik \cite{CKM}.

 Let $X = W//G$ be a GIT quotient. The moduli space of quasimaps from $\mathbb{P}^1$ to $X$, denoted by $QM(X)$, carries evaluation maps $ev_p: QM(X) \to [W/G]$ for each $p \in \mathbb{P}^1$, where $[W/G]$ is the quotient stack. A quasimap $f$ is said to be stable at a point $p \in \mathbb{P}^1$ if $ev_p(f) = f(p)\in X \subset [W/G]$; it is called stable if this property holds for all but finitely many points $p \in \mathbb{P}^1$. 

Each quasimap has a well-defined degree, and the moduli space decomposes accordingly as a countable union:
\[
QM(X)= \bigsqcup_d QM^d(X),
\]
 where $QM^d(X)$ denotes the stack of stable quasimaps of degree $d$. For spaces like $T^*G(r,n)$, the degree lies in $\mathbb{Z}$. Each $QM^d(X)$ carries a perfect obstruction theory, which allows one to define the virtual structure sheaf $\mathcal{O}_{d}^\mathrm{vir}$ and the virtual normal bundle $N_{d}^{\mathrm{vir}}$; more details for $T^*G(r,n)$ appear in \Cref{section 7}.

Using this setup, \cite{CKM} defined the cohomological $I$-function, an analogue of the $J$-function in the quasimap setting. The $K$-theoretic $I$-function is defined in \cite{wen2019ktheoreticifunctionvthetamathbfg} as follows. Let $F_0 \subset QM^d(X)$ denote the locus of quasimaps that are constant on $\mathbb{P}^1 \setminus \{0\}$; these are automatically stable at $\infty \in \mathbb{P}^1$. The $I$-function is defined by:
\begin{equation} \label{def_I}
I: = \sum_{d \geq 0} Q^{d} ev_{{\infty},*}\Big(\frac{\mathcal{O}_{d,F_0}^{\mathrm{vir}}}{\lambda_{-1}(N_{\mathrm{vir},d,F_0}^{*})}\Big),
\end{equation}
where $\mathcal{O}_{d,F_0}^{\mathrm{vir}}$ and $N_{\mathrm{vir},d,F_0}^{*}$ denote the induced virtual structure sheaf and virtual conormal bundle on $F_0$, and the (Hirzebruch) $\lambda_{y}$ class is defined in $K$-theory by: 
\[ \lambda_{y}(E) := \sum_{i} y^i[\wedge^i E]. \]
The Grassmannian $G(r,n)$ can be realized as the GIT quotient $M_{r,n}//_{\det} GL_r$, the space of $ r \times n$ matrices divided by $GL_r$, with the stability condition given by the determinant.
The moduli space of degree $d$ quasimaps to $G(r,n)$ identifies with the Quot scheme $Quot_{\mathbb{P}^1,d}(\mathbb{C}^n,r)$, a smooth projective variety parametrizing rank $n-r$ coherent quotient sheaves of degree $d$ on $\mathbb{P}^1$:
\[
\mathbb{C}^n \otimes \mathcal{O}_{\mathbb{P}^1} \to F \to 0.
\]
Under this identification, Taipale \cite{Taipale} showed that the quasimap $I$-function coincides with the $J$-function of $G(r,n)$. When we discuss $G(r,n)$ throughout the paper, we may use the $J$-function and the $I$-function interchangeably.

Another key object for Nakajima quiver varieties, introduced by Okounkov \cite{MR3752463}, is the quasimap vertex function, a generating function for equivariant counts of stable quasimaps from $\mathbb{P}^1$ to Nakajima quiver varieties. For $T^*G(r,n)$, it is defined by: 
 \[
 \sum_{d \geq 0}ev_{{\infty},*}(QM^d_{\infty}, \hat{\mathcal{O}}^{{\mathrm{vir}}}_d) Q^d \in K_{T \times \mathbb{C}^*_q \times \mathbb{C}^*_{\hbar}}(T^*G(r,n))_{\mathrm{loc}}[[Q]],
\]
where $QM^d_{\infty}(X)$ parametrizes degree $d$ quasimaps stable at $\infty$, and $\hat{\mathcal{O}}_d^{{\mathrm{vir}}}$ is the twisted virtual structure sheaf defined by:
\[
\hat{\mathcal{O}}_d^{{\mathrm{vir}}} = {\mathcal{O}^{{\mathrm{vir}}}_d} \otimes \mathcal{K}_{\mathrm{vir}}^{1/2},
\]
with $\mathcal{K}_{\mathrm{vir}}^{1/2}$ a square root of the 
virtual canonical bundle $\mathcal{K}_{\mathrm{vir}}$.
By the localization theorem, the quasimap vertex function can 
also be expressed as:
\[
\sum_{d} \chi(\hat{\mathcal{O}}^{{\mathrm{vir}}}_d)Q^d,
\]
which justifies interpreting it as a generating series of equivariant counts.

Pushkar, Smirnov, and Zeitlin \cite{PSZ} computed an explicit formula for the quasimap vertex function $V_x(t;q,Q,\hbar)$ localized at each fixed point $x \in T^*G(r,n)$. Writing:
\[
V_x(t;q,Q,\hbar) = \sum_{\substack{d_1+\dots+d_r=d \\ d_i \ge 0}}V_x^d(t;q,\hbar) Q^d , 
\]
the coefficients are:
\[
V_x^d(t;q,\hbar) = q^{nd/2}\prod_{i,j=1}^r \{\frac{t_j}{t_i}\}_{d_i-d_j}^{-1} \cdot \prod_{i=1}^{r}\prod_{j=1}^{n}\{\frac{t_j}{t_i}\}_{d_i} ,
\]
where:
\[
\{x\}_d = \frac{(\frac{\hbar}{x})_d}{(\frac{q}{x})_d}(-q^{\frac{1}{2}}\hbar^{\frac{-1}{2}})^d, \quad \quad \phi(x) = \prod_{i=0}^{\infty}(1-xq^i), \quad   \quad (x,q)_d = \frac{\phi(x)}{\phi(q^dx)}\text{.}
\]
Here, $\hbar$ is the weight of the $\mathbb{C}^*$-action scaling the cotangent direction of $T^*G(r,n)$, $t=(t_1,\dots,t_n)$ denotes the vector of weights of the torus $T$, and $q$ is the weight of the $\mathbb{C}^*_q$-action on the domain $\mathbb{P}^1$ of quasimaps.
\begin{definition}
    The twisted $I$-function of $T^* G(r,n)$ is obtained from the previously defined $I$-function \eqref{def_I} by replacing $\mathcal{O}_{d,F_0}^{\mathrm{vir}}$ with its twisted version:
\begin{equation}
I^{tw} := \sum_{d \geq 0} Q^d ev_{{\infty},*}\Big(\frac{\hat{\mathcal{O}}_{d,F_0}^{\mathrm{vir}}}{\lambda_{-1}(N_{\mathrm{vir},d,F_0}^{*})}\Big).
\end{equation}
\end{definition}

We show that, for $T^*G(r,n)$, the twisted $I$-function and the quasimap vertex function coincide. While this result is implicit in the literature, to our knowledge, it has not been explicitly stated or proved. It is the starting point of our contribution.
\begin{thm}[See \Cref{prop:V=I} below.]
    For $T^*G(r,n)$, the twisted $I$-function of $T^* G(r,n)$ is equal to the quasimap vertex function.
\end{thm}

Next, we relate the $K$-theoretic $I$-function of $G(r,n)$ to the quasimap vertex function of $T^*G(r,n)$. Taking the limit $\hbar \to 0$ of the localized quasimap vertex function at each fixed point of $T^*G(r,n)$ recovers the localization of the $I$-function of $G(r,n)$ at the same fixed point. Our main contribution is to provide a reverse procedure: starting from the $I$-function of $G(r,n)$, we recover the quasimap vertex function on $T^*G(r,n)$ via a certain $K$-theoretic operator called {\bf balancing}. In other words, we provide an algebraic procedure that lifts the $I$-function to the quasimap vertex function, which we explain below.

Fix a collection of  $T \times \mathbb{C}^*_q$-modules $\{V_i : \lambda_{-1}V_i \neq 0 \}_{i \in \Lambda}$, and let $A \subset K_{T \times \mathbb{C}^*_q}(\mathrm{pt})$ denote the multiplicative set generated by $\lambda_{-1}V_i$, $i \in \Lambda$. Denote the localization of $ K_{T \times \mathbb{C}^*_q}(\mathrm{pt})$ with respect to $A$ by $K_{T \times \mathbb{C}^*_q}(\mathrm{pt})_{\mathrm{loc}}$. Define the subset $U \subset K_{T \times \mathbb{C}^*_q}(\mathrm{pt})_{\mathrm{loc}}$ by:
\[
 U=\{\prod_{i=1}^m\frac{1}{u_i} : u_i =\lambda_{-1}V_i \quad \text{for distinct $i \in \Lambda$, $m < \infty$}\}.
 \]
 We define the balancing operator:
 \[
 \mathcal{B}_y : \mathbb{N}\langle U \rangle \to K_{T \times \mathbb{C}^*}(\mathrm{pt})[y],
 \]
  where on $U$ is defined by:
 \[
 \frac{1}{\prod_{i=1}^m\lambda_{-1}V_i} \mapsto  \frac{\prod_{i=1}^m\lambda_{y}V_i}{\prod_{i=1}^m\lambda_{-1}V_i},
 \]
and expand additively on $\mathbb{N}\langle U \rangle$.

 We extend $\mathcal{B}_y$ to power series in $U[[Q]]$, where $Q$ is a formal variable:
\[
\mathcal{B}_y(\sum_{d \geq 0}A_d Q^d) = \sum_{d\geq 0} \mathcal{B}_y(A_d) Q^d, \quad \text{ $A_d \in U$}.
\]
  
 Let $X$ be a $T$-variety such that $\big| X^T\big| < \infty$, equipped with an additional trivial $\mathbb{C}^*$-action. This means that:
 \[
 K_{T \times \mathbb{C}^*}(X) = K_T(X)[q,q^{-1}],
 \]
  where $q$ generates $K_{\mathbb{C}^*}(\mathrm{pt})=\mathbb{C}[q,q^{-1}]$.

\begin{definition}
Let $\gamma$  be a class in  $K_{T}(X)[[q]]$ such that for each fixed point $x \in X^T$, $\gamma_{\arrowvert_x}$ is a rational function in $\mathbb{N}\langle U\rangle$. We call such a class a \textbf{U-good} class. We define:
\begin{itemize}
    \item {\bf $\lambda_y$-balanced $K$-theoretic class} of $\gamma$ at $x$ to be $\mathcal{B}_y(\gamma_{{\arrowvert}x})$.
    \item {\bf $\lambda_y$-balanced $K$-theoretic class} of $\gamma$, denoted by $\mathcal{B}_y(\gamma)$ whenever it exists, to be the class satisfying $\mathcal{B}_y(\gamma)_{{\arrowvert}x} = \mathcal{B}_y(\gamma_{{\arrowvert}x})$.
\end{itemize}
\end{definition}

 We now define a balancing operator on a certain class of $q$-difference operators. Let
 \[
 S:=R[q^{\pm Q \partial_{Q}}],
 \]
 be the algebra of $q$-difference operators generated over the ring:
 \[
 R :=\mathbb{C}[P, t_1^{-1},\dots,t_n^{-1},q,Q].
 \]
 As discussed in \Cref{operator for operators}, any element of $S$ has a canonical expansion:
 \[
       \sum_{b \in \mathbb{Z}}\sum_{a \in \mathbb{N}} c_a(t^{-1}, P,q)(q^{Q \partial_{Q}})^{b}Q^a.
\]

 Define a subset $\mathcal{C} \subset S$
        by:
        \[
        \mathcal{C}:=\Big\{\prod_{\text{finite}}(1-M) \Big| \text{$M$ is a monomial of positive degree in $q^{\pm Q \partial_{Q}}$}   \Big\}.
        \]
        Define the balancing operator $\mathcal{H}_{\hbar}$ on $\mathcal{C} \oplus R$ as follows:
        \begin{itemize}
            \item On $\mathcal{C}$:
         \begin{equation*}
          \mathcal{H}_{\hbar}\left(\prod_{\text{finite}}(1-M_i)\right) = \frac{\prod_{\text{finite}}(1-M_i)}{\prod_{\text{finite}}(1-\hbar M_i)} := \prod_{finite}(1-\hbar M_i)^{-1}\cdot(1-M_i),
       \end{equation*}
       \item On $R$: $\mathcal{H}_{\hbar}$  is identity.
 \end{itemize}

  We apply these constructions to $X = G(r,n)$. Let the $I$-function of $G(r,n)$ be:
\begin{equation} \label{intro}
I(t;q,Q) = \sum_{d \geq 0} I_d(t;q) Q^d.
\end{equation}

We will show that each coefficient $I_d(t;q)$ is a U-good class for a natural choice of $T \times \mathbb{C}^*_q$-modules. Our main theorem establishes that the quasimap vertex function of $T^*G(r,n)$, after absorbing certain factor into the quantum parameter $Q$, can be obtained purely algebraically by applying the balancing operator to the $I$-function of $G(r,n)$. Moreover, the balancing of classes and $Q$-difference operators is compatible in certain cases:
\begin{thm} 
\begin{enumerate}
    \item 
 Let  $x=\langle e_{i_1}, \dots, e_{i_r} \rangle$ be a fixed point of $T^*G(r,n)$. Then, upon substituting $y=-q^{-1}\hbar$, we obtain: 
    \[
    V_x(t;q,Q,\hbar) = \mathcal{B}_{-q^{-1}\hbar}(I(t;q,Q)_{\arrowvert_{x}}) ,
    \]
    where the coefficients of $Q^d$ in both functions are equal for every $d \geq 0$:
    \[
    V_x^d(t;q,\hbar) = \mathcal{B}_{-q^{-1}\hbar}(I_d(t;q)_{\arrowvert_{x}}).
    \]
    \item Let $\mathcal{I} = \sum_{d \geq 0}\frac{1}{\prod_{i=1}^n(qa_i)_d}Q^d$, where $a_i \in K_T(X)$ satisfy $\prod_{i=1}^n(1 - a_i) = 0$, and the restriction of each $a_i$ to any $x \in X^T$ is a monomial. Define the operator $\mathcal{D} =\prod_{i=1}^n (1 - a_iq^{Q\partial _Q}) - Q$. Then, for an appropriate collection of $T \times \mathbb{C}^*_q$-modules, we have:
    \begin{itemize}
        \item $\mathcal{D}$ annihilates $\mathcal{I}$, that is, $\mathcal{D} \mathcal{I} = 0$.
        \item $\mathcal{H}_y(\mathcal{D})$ annihilates $\mathcal{B}_y(\mathcal{I})$, that is; $\mathcal{H}_y(\mathcal{D}) \mathcal{B}_y(\mathcal{I})=0$.
    \end{itemize}

    \end{enumerate}
 \end{thm}
\begin{example}

 We now illustrate the theorem for $X = \mathbb{P}^{1}$. The $I$-function of $\mathbb{P}^{1}$ is:
 \[
 I(t;q,Q)=\sum_{d \geq 0}I_d(t;q)Q^d=\sum_{d \geq 0}\frac{Q^d}{\prod_{i=1}^n(qPt_i^{-1})_d},
 \]
 where $P = \mathcal{O}(-1)$ satisfies $(1-Pt_1^{-1})(1-Pt_2^{-1}) = 0$. The two torus-fixed points are $0 = [1:0]$ and $\infty = [0:1]$. Localization of the coefficient $Q^d$, denoted by $I_d(t;q)$, at $0$ is given by:
 \[
 \frac{1}{(1-q)\cdots(1-q^d)}\cdot\frac{1}{(1-q\frac{t_1}{t_2})\cdots(1-q^d\frac{t_1}{t_2})}.
 \]
Applying the balancing operator $\mathcal{B}_y$ yields (See \Cref{lambda calculation} for the collection of modules we need to consider):
\[
\frac{(1+yq)\cdots(1+yq^d)}{(1-q)\cdots(1-q^d)}\cdot\frac{(1+yq\frac{t_1}{t_2})\cdots(1+yq^d\frac{t_1}{t_2})}{(1-q\frac{t_1}{t_2})\dots(1-q^d\frac{t_1}{t_2})}.
\]
 Substituting $y = -q^{-1}\hbar$ gives:
 \[
 \frac{(1-\hbar )\cdots(1-\hbar q^{d-1})}{(1-q)\cdots(1-q^d)}\cdot\frac{(1-\hbar \frac{t_1}{t_2})\cdots(1-\hbar q^{d-1}\frac{t_1}{t_2})}{(1-q\frac{t_1}{t_2})\cdots(1-q^d\frac{t_1}{t_2})},
 \]
 which matches the localization of $V^d(t;q,\hbar)$ at $0$. The $Q$-difference operators annihilating the $I$-function of $\mathbb{P}^1$ are also well known. We will see their balancing in \Cref{Balancing of operators for projective space}.
 \end{example}
 Although the proof is purely algebraic, the underlying geometric insight is that
 balancing each localized coefficient $I_d$ of the $I$-function at a fixed point of $G(r,n)$ recovers the localization of the corresponding coefficient $V_d$ of the quasimap vertex function at the same fixed point of $T^*G(r,n)$, after substituting $y = -q^{-1}\hbar$.
 
 The central idea of the proof lies in the relationship between the localization of the $I$-function and the Quot scheme. As shown in \Cref{Khaste}, each localized coefficient $I_d(t;q)$ at a fixed point of $G(r,n)$ can be written as:
 \[
 \sum_x\frac{1}{\lambda_{-1}T^*_{x}Quot_{\mathbb{P}^1,d}(\mathbb{C}^n,r)},
 \]
 where the sum runs over certain $T \times \mathbb{C}^*$-fixed points $x$ of the Quot scheme. In \Cref{fixed-point lemma}, we compute the weight space decomposition of the cotangent space at these fixed points, and in \Cref{lambda calculation}, we determine the {\bf $\lambda_y$-balanced $K$-theoretic class} of each $I_d(t;q)$.  After defining the stable quasimap space and quasimap vertex function, we complete the proof in \Cref{Main results}. In effect, balancing amounts to twisting each fixed-point contribution by the cotangent weights, producing exactly the vertex contribution.

The quantum $K$-ring of Nakajima quiver varieties was defined in \cite{PSZ} using the moduli space of stable quasimaps. The relationship between the quantum $K$-theory of a variety and that of its cotangent bundle was previously explored for the full flag variety $\mathrm{Fl}(n)$ in \cite{MR4308933}, where the authors relate the quantum $K$-ring of $\mathrm{Fl}(n)$ and that of $T^*\mathrm{Fl}(n)$ by taking an appropriate limit. Our result establishes a direct connection between two key generating functions—the $I$-function of $G(r,n)$ and the quasimap vertex function of $T^*G(r,n)$—through the balancing operator.
 
 For general Grassmannian $G(r,n)$, difference operators annihilating the $I$-function are not known. However, by applying the abelian/nonabelian correspondence, we can abelianize $G(r,n)$ to $(\mathbb{P}^{n-1})^r$. In this abelian setting, an appropriate twist can be defined, allowing both a value of the big $J$-function, denoted by $\overline{J}^{tw}$, and the $q$-difference operators it satisfies to be computed. In \Cref{Bethe section}, we use the idea of the balancing to construct new $q$-difference operators that are satisfied by the $\lambda_y$ $K$-theoretic class of $\overline{J}^{tw}$. Using the abelian/nonabelian map, we prove that these equations, after a suitable change of variables, coincide with the Bethe Ansatz equations, central objects in integrable systems, which are expected to be the relations in the quantum $K$-theory of $T^*G(r,n)$. 

\textbf{Acknowledgments}. I am very grateful to my advisor Leonardo Mihalcea for introducing me to this area and sharing ideas with me and many related discussions. 

I would like to thank Leo Herr, Irit Huq-Kuruvilla, and Andrey Smirnov for many related discussions. This project was supported by NSF grant DMS-2152294.

\section{Algebra Preliminaries}
In the following three sections, we briefly present the purely algebraic perspective of our results. Later, we show that these algebraic constructions align with the geometric framework discussed throughout the paper.
\subsection{Balancing classes}   Let  $R:=\mathbb{C}[{t_1^{\pm 1}},\dots,{t_n ^{\pm 1}},q^{\pm1}]$ be the ring of  (Laurent) polynomials in the variables ${t_i}$, $1 \leq i \leq n$, and $q$. Define a multiplicative set $\mathcal{A} \subset R$ by:
\[
\mathcal{A} = \Bigl\{\prod_{\text{finite}}(1-P) : \text{$P$ is a monomial in $ R,\; P \ne 1$} \Bigr\},
\]
and define $U \subset R_\mathcal{A}$ by:
\[
U = \Bigl\{\frac{1}{u}: u \in \mathcal{A} \Bigr\}.
\]
Let $y$ be a formal variable. Define the balancing operator $\mathcal{B}_y$ on $U$ by:
\[
\mathcal{B}_y(\frac{1}{\prod_{\text{finite}}(1-P)}) = \frac{\prod_{\text{finite}}(1+yP)}{\prod_{\text{finite}}(1-P)}.
\]
Extend $\mathcal{B}_y$ additively to the monoid over $\mathbb{N}$ generated by $U$, and linearly to the ring of power series in the variable $Q$, namely $U[[Q]]$. More precisely, if $\sum_{d \geq 0}A_d Q^d  \in U[[Q]]$, we set:
\[
\mathcal{B}_y(\sum_{d \geq 0}A_d Q^d) = \sum_{d\geq 0} \mathcal{B}_y(A_d)Q^d.
\]

Consider the following function, which appears in the study of the geometry of $G(r,n)$ (see \Cref{sec:Quot scheme} and \Cref{J-function section} for details):
\[
I(t;q,Q) = \sum_{d=d_1+\dots+d_r} I_d(t;q) Q^d   \quad \/;
\]
where $d_i \geq 0$ and:
\[
I_d(t;q)=\frac{1}{\prod_{i=1}^r \prod_{j=r+1}^n\prod_{m=1}^{d_i}(1-\frac{t_i}{t_j}q^{m})(1-q^m)\prod_{i,j=1}^r \prod_{m=-d_i}^{-d_i+d_j-1}(1-\frac{t_i}{t_j}q^{-m})}.
\]
In \Cref{sec:Quot scheme}, we will see that this function appears as a localization of a $K$-theoretic class on $G(r,n)$, namely the $J$-function (which coincides with the $I$-function).

Now, if we apply $\mathcal{B}_y$ to this power series, we define the {\em balanced $I$-function}:
\[
I(t;q,Q,y):=\mathcal{B}_y(I(t;q,Q)) .\]
Writing $\mathcal{B}_y I(t;q,Q)  = \sum_{d \geq 0} I_d(t;q,y)Q^d$
then $I_d(t;q,y):=\mathcal{B}_y(I_d(t;q))$, and we compute:
\[
\mathcal{B}_y(I_d(t;q)) = \frac{\prod_{i=1}^r \prod_{j=r+1}^n\prod_{m=1}^{d_i}(1+y\frac{t_i}{t_j}q^{m})(1+y q^m)\prod_{i,j=1}^r \prod_{m=-d_i}^{-d_i+d_j-1}(1+y\frac{t_i}{t_j} q^{-m})}{\prod_{i =1}^r \prod_{j =r+1}^n\prod_{m=1}^{d_i}(1-\frac{t_i}{t_j}q^{m})(1-q^m)\prod_{i,j=1}^r \prod_{m=-d_i}^{-d_i+d_j-1}(1-\frac{t_i}{t_j}q^{-m})}.
\]

Another function that appears in the study of the geometry of $T^*G(r,n)$, and we discuss the details in \Cref{vertex function}, is the {\em vertex function}:
\[
V= V(t;q,Q,\hbar) = \sum_{d_1+\dots+d_r=d}V_d(t;q,\hbar) Q^d \quad \/;
\]
where $d_i \geq 0$ and
\[
V_d=V_d(t;q,\hbar) = q^{nd/2}\prod_{i,j=1}^r \{\frac{t_j}{t_i}\}_{d_i-d_j}^{-1}.\prod_{i=1}^{r}\prod_{j=1}^{n}\{\frac{t_j}{t_i}\}_{d_i} \quad \/;
\]
\[
\]
 \[
   \{x\}_d = \frac{(\frac{\hbar}{x})_d}{(\frac{q}{x})_d}.
 \]
 Here $(x)_d$ is the Pochhammer symbol defined by: 
\begin{equation}\label{Pochammer Symbols}
(x)_d:=\frac{\phi(x)}{\phi(xq^d)}  \quad \/, \quad \phi(x) = \prod_{i=0}^{\infty}(1-xq^i).
\end{equation}

In \Cref{Main results}, we will prove the following theorem:

    \begin{thm}\label{thm:V=balI} The quasimap vertex function $V$ and the balanced $I$-function are equal, up to the change of variables
   $y = -q^{-1} \hbar$:
   \[ V(t;q,Q,\hbar)= I(t; q, Q,-q^{-1} \hbar)\]
   Equivalently, for each $d \ge 0$, 
   \[ V_d(t;q,\hbar)= I_d(t; q, -q^{-1} \hbar) \/.\]
\end{thm}
\begin{proof} This is a tedious algebraic check.\end{proof}
In \Cref{Main results}, we will re-prove this theorem after interpreting each of the terms
of $J$ and $V$ in terms of equivariant localization. The balanced $I$-function will arise from balancing 
localization terms on the Quot Scheme of the Grassmannian, while the vertex function $V$ arises from 
the localization data on the moduli space of quasimaps to $T^*G(r,n)$. The equality of the two functions
will be a consequence of the {\em termwise} equality of localization data.

\subsection{Balancing Operators}\label{operator for operators}
In this section, we develop algebraic constructions for operators acting on $K$-theoretic classes. We first define a polynomial ring $R$ and an algebra $S$ over $R$. Next, we define an operator on a certain subset $\mathcal{C}$ of $S$. The variables and the subset $\mathcal{C}$ are chosen so that the operators annihilating the $J$-functions can be expressed using them. 

        Let $R := \mathbb{C}[P,t_1^{-1},\dots,t_n^{-1},q,Q]$ be the ring generated by variables $P$, $t^{-1}=(t_1^{-1},\dots,t_n^{-1})$, together with the independent variables $q$, $Q$. Define $S$ as the algebra over $R$ generated by $q$-difference operators $q^{\pm Q \partial_{Q}}$:
        \[
        S:=R[q^{\pm Q \partial_{Q}}],
       \]
       where $q^{\pm Q \partial_{Q}}$ acts on functions of $Q$ by:
        \begin{equation} \label{diff op}
        f(Q) \mapsto f(qQ).
        \end{equation}
        The action of $P$, $t_i^{-1}$'s and $q$ on difference operators is by multiplication, while $Q$ acts only on $q^{Q \partial_{Q}}$ as follows:
        \[
        Qq^{Q \partial_{Q}} = q^{-1}q^{Q \partial_{Q}}Q.
        \]
        Therefore, for monomials we have:
        \[
Q^{a} (q^{Q \partial_{Q}})^{b}
=  q^{-a-b} (q^{Q \partial_{Q}})^{b} Q^{a}.
\]

        Using this rule, we can write any element of $S$ in a canonical way:
        \[
        \sum_{b \in \mathbb{Z}}\sum_{a \in \mathbb{N}} c_a(t^{-1}, P^{\pm 1},q)(q^{Q \partial_{Q}})^{b}Q^a,
        \]
        where we mean that the variable $Q$ can be moved to the right of each monomial.
        Define $\mathcal{C} \subset S$
        by:
        \[
        \mathcal{C}:=\{\prod_{\text{finite}}(1-M): \text{$M$ is a monomial of positive degree in $q^{\pm Q \partial_{Q}}$}\}.
        \]
         
       Let $\hbar$ be a formal variable. For any difference operator $1-M \in \mathcal{C}$, we define a new difference operator $1-\hbar M$ as an element of the ring of power series $T=R[q^{Q_1 \partial_{Q_1}},\dots,q^{Q_r \partial_{Q_r}}][[\hbar]]$. Since the constant term is 1, this difference operator has an inverse in $T$. More precisely, for a monomial of positive degree $M \in S$, we define:
       \[
       \frac{1}{(1-\hbar M)}:=(1-\hbar M)^{-1} = \sum_{i=0}^{\infty}\hbar^{i}M^i,
       \]
       which is an element of $T$.
       
       Define the balancing operator $\mathcal{H}_{\hbar}$ on $\mathcal{C} \oplus R$, where on $\mathcal{C}$ is given by:
       \begin{equation}\label{H-operator}
          \mathcal{H}_{\hbar}\left(\prod_{\text{finite}}(1-M_i)\right) = \frac{\prod_{\text{finite}}(1-M_i)}{\prod_{\text{finite}}(1-\hbar M_i)} = \prod_{\text{finite}}(1-\hbar M_i)^{-1}\prod_{\text{finite}}(1-M_i) ,
       \end{equation}
       and on $R$ is the identity:
       \[
       \mathcal{H}_{\hbar}(L) = L,
       \]
       for $L \in R$.

\section{Equivariant K-theory}
\subsection{Preliminaries}
In this section, we review some necessary facts about the equivariant $K$-theory of smooth projective varieties. See for example, \cite{NIE, CG}.

Let $X$ be a smooth projective variety with an action of a torus $T$. We denote by $K_T(X)$ the Grothendieck ring generated by symbols $[E]$, where
$E$ is a $T$-equivariant vector bundle on $X$, subject to the relations $[E] = [F] + [G]$ whenever there exists a short exact sequence of $T$-equivariant vector bundles:  
\[
0 \to F \to E \to G \to 0.
\]
The additive structure on $K_T(X)$ is given by the direct sum of vector bundles, and the multiplicative structure is given by the tensor product.

Since $X$ is smooth, any $T$-equivariant coherent sheaf admits a finite resolution by locally free sheaves. Hence, $K_T(X)$ coincides with the Grothendieck ring of $T$-equivariant coherent sheaves on $X$, denoted by $K_0^T(X).$
In the special case where $X=\mathrm{pt}$, we have $K_T(X) = \mathbb{Z}[{t_1^{\pm 1}},\dots,{t_n^{\pm 1}}]$, where $T \cong (\mathbb{C}^*)^n$ and ${t_i}$ are characters of $T$.

Let $f : X \to Y$ be a $T$-equivariant morphism between smooth $T$-varieties. Then there is a pullback homomorphism:
\[
f^*: K_T(Y) \to K_T(X), \quad [E] \mapsto [f^*E],
\]
where $E$ is a $T$-equivariant vector bundle on $Y$. Note that smoothness of $Y$ is not required for this definition.
 
 If $f$ is proper, there is also a push forward map:  
 \[
 f_*:K_T(X) \to K_T(Y),
 \]
 defined by
\begin{align*}
f_*([\mathcal{F}]) = \sum_i (-1)^iR^i [f_*(\mathcal{F})],
\end{align*}
where $\mathcal{F}$ is a $T$-equivariant coherent sheaf on $X$, and $R^if_*(\mathcal{F})$ denotes the higher direct image functor.

If $Y = \mathrm{pt}$, the push forward reduces to the $T$-equivariant Euler characteristic of $\mathcal{F}$:
\begin{align*}
    f_*(\mathcal{F}) = \chi(X,\mathcal{F}) := \int_X \mathcal{F} = \sum_i (-1)^i[H^i(X,\mathcal{F})] ,
\end{align*}
where each $H^i(X,\mathcal{F})$ carries a natural $T$-representation structure.
The $T$-equivariant $K$-theoretic Poincar\'e pairing is defined as follows:
\begin{align*}
    &K_T(X) \times K_T(X) \to K_T(\mathrm{pt}); \\&
    \langle\mathcal{E} , \mathcal{F}\rangle \mapsto \int_X \mathcal{E} \otimes_{\mathcal{O}_X} \mathcal{F}.
\end{align*}

In $K$-theory, the (Hirzebruch) $\lambda_y$ class is defined by: 
\[ \lambda_y(E) := 1 + y [E] + y^2 [\wedge^2 E] + \dots + y^e  [\wedge^e E] \in \K_T(X)[y]. \] 
This class was introduced by Hirzebruch \cite{MR1335917} in order to help with the 
formalism of the Grothendieck-Riemann-Roch theorem. It may be thought of as the 
$K$-theoretic analogue of the (cohomological) Chern polynomial: 
\[ c_y(E)= 1+ c_1(E) y + \ldots + c_e(E) y^e\] 
of the bundle $E$. 
The $\lambda_y$ class is multiplicative with respect to short exact sequences, that is, if: 
\[ \begin{tikzcd} 0 \arrow[r] & E_1 \arrow[r] & E_2 \arrow[r] & E_3 \arrow[r] & 0 \end{tikzcd} \] 
is such a sequence of vector bundles, then: 
\[ \lambda_y(E_2) = \lambda_y(E_1) \cdot \lambda_y(E_3) \/; \] 
cf.~\cite{MR1335917}. 
A particular case of this construction is when $V$ is a (complex) vector space with an action 
of a complex torus $T$, 
and with weight decomposition 
$V = \oplus_i V_{\mu_i}$, where each $\mu_i$ is a weight.
The {\em character} of $V$ is the element $\ch_T(V):=\sum_i \dim V_{\mu_i} {\mu_i}$, regarded  in $\K_T(\mathrm{pt})$.
The  $\lambda_y$-class of $V$ is the element 
$\lambda_y(V) =\sum_{i \ge 0} y^i ch_T(\wedge^i V) \in \K_T(\mathrm{pt})[y]$.
From the multiplicativity property of the $\lambda_y$-class it follows that 
\[\lambda_y(V) = \prod_i (1+y {\mu_i})^{\dim V_{\mu_i}} \/;\] 
see \cite{MR1335917}.

A key computational tool in equivariant $K$-theory is the Atiyah-Bott Localization formula from \cite{NIE}, recalled next.
Let $X$ be a proper, smooth scheme with a $T$-action such that the fixed-point set $X^T$ is finite. 
Let $\iota:X^T \to X$ 
be the inclusion, and consider the pullback map: 
$\iota^*: K_T(X) \to K_T(X^T)$. For $x \in X$, denote $\alpha|_x := \iota_x^*(\alpha)$, where $\iota_x: \{x\} \hookrightarrow X$ is the inclusion. Then the following holds.
\begin{thm}\label{Localization-Theorem} Let $\alpha \in K_T(X)$ be a class in the equivariant $K$ theory of $X$. Then,
\[
   \alpha  = \sum_{x \in X^T}\frac{\alpha\arrowvert_x}{\lambda_{-1}(T_x^*X)}[x] . 
\]
In particular, 
\begin{align*}
    \int_{X} \alpha = \sum_{x \in X^T}\frac{\alpha\arrowvert_x}{\lambda_{-1}(T_x^*X)}
\end{align*}
\end{thm}
\subsection{Balanced $K$-theoretic classes} \label{Balanced class} In this section, we introduce the notion of balanced classes, a central concept used throughout the paper. This notion plays a crucial role in relating the $K$-theoretic functions of $T^*G(r,n)$ and $G(r,n)$ via multiplication by $\lambda_y$-classes. 
The notion of balanced classes has been previously discussed in works such as \cite{Givental:2020, MR3752463} . 
We take a different approach, 
in which classes are balanced via multiplication with $\lambda_y$-classes.

Fix a set of $T \times \mathbb{C}^*_q$-modules $\{V_i : \lambda_{-1}V_i \neq 0 \}_{i \in \Lambda}$ for some index set $\Lambda$, and define $A \subset K_{T \times \mathbb{C}^*_q}(\mathrm{pt})$ to be the multiplicative set generated by the elements $\lambda_{-1}V_i$, for each $i \in \Lambda$.

Let $K_{T \times \mathbb{C}^*_q}(\mathrm{pt})_{loc}$ denote the localization of $K_{T \times \mathbb{C}^*}(\mathrm{pt})$ at $A$. Define $U \subset K_{T \times \mathbb{C}^*}(\mathrm{pt})_{loc}$ as:
\[
 U=\{\prod_{i=1}^m\frac{1}{u_i} : u_i =\lambda_{-1}V_i \quad \text{for distinct i $\in \Lambda$, $m < \infty$}\}.
 \]
 Define  
\[
\mathcal{B}_y \colon U \to K_{T \times \mathbb{C}^*}(\mathrm{pt})_{\mathrm{loc}}[y]
\]
 by:
 \[
 \frac{1}{\prod_{i=1}^m\lambda_{-1}V_i} \mapsto  \frac{\prod_{i=1}^m\lambda_{y}V_i}{\prod_{i=1}^m\lambda_{-1}V_i}.
 \]
 Extend $\mathcal{B}y$ additively to the monoid $\mathbb{N}\langle U \rangle$ and naturally to power series in $U[[Q]]$: for $\sum{d \ge 0} A_d Q^d \in U[[Q]]$, set:
 \[
 \mathcal{B}_y\Bigl(\sum_{d \ge 0} A_d Q^d\Bigr) = \sum_{d \ge 0} \mathcal{B}_y(A_d) Q^d.
 \]
 Let $X$ be a $T$-variety with $| X^T| < \infty$. We equip $X$ with an additional $\mathbb{C}^*$-action that acts trivially on $X$. Then, $K_{T \times \mathbb{C}^*}(X) = K_T(X)[q,q^{-1}]$, where we interpret $q$ geometrically as $K_{\mathbb{C}^*}(\mathrm{pt}) = \mathbb{C}[q,q^{-1}]$.

\begin{definition}\label{balancing definitions}
\begin{enumerate}
    \item  Let $\gamma \in K_{T}(X)[[q]]$. We say $\gamma$ is \textbf{U-good} if, for every fixed point $x \in X^T$, $\gamma_{\arrowvert_x}$ is a rational function in $\mathbb{N}\langle U \rangle$.
    \item For a U-good class $\gamma$, we define its {\bf $\lambda_y$-balanced $K$-theoretic class at the fixed point $x$} as $\mathcal{B}_y(\gamma_{{\arrowvert}x})$.
    \item For a U-good class $\gamma$, we define its {\bf $\lambda_y$-balanced $K$-theoretic class}, denoted by $\mathcal{B}_y(\gamma)$ whenever it exists, to be the class such that $\mathcal{B}_y(\gamma)_{{\arrowvert}x} = \mathcal{B}_y(\gamma_{{\arrowvert}x})$, for all $x \in X^T$
    \end{enumerate}
\end{definition}

\begin{example}
    Let $T=(\mathbb{C}^*)^2$ act diagonally on $\mathbb{P}^1$ by scaling coordinates with weights $t_1$ and $t_2$. The group $\mathbb{C}^*_q$ acts by $q \cdot [x:y] = [qx:y]$. Let $q\mathcal{O}(-1) \in K_{T \times \mathbb{C}^*}(\mathbb{P}^1)$. By restricting this class to the fixed points $0$ and $\infty$, we obtain:
    \[
     q\mathcal{O}(-1)_{{\arrowvert}_0} = qt_1, \quad q\mathcal{O}(-1)_{{\arrowvert}_{\infty}} = qt_2.
    \]
     Let $V_1, V_2$ be one-dimensional $T \times \mathbb{C}^*_q$-modules defined by the weights $qt_1$ and $qt_2$, respectively. 
    Then $U$ is defined by $\{ \frac{1}{1-qt_1}, \frac{1}{1-qt_2} ,\frac{1}{(1-qt_1)(1-qt_2)},1 \}$. Let $\gamma = \frac{1}{1-q\mathcal{O}_{\mathbb{P}^1}} \in K_{T}(\mathbb{P}^1)[[q]]$. We have:
    \[
    \gamma_{{\arrowvert}_{0}} = \frac{1}{1-qt_1}; \quad \gamma_{{\arrowvert}_{\infty}} = \frac{1}{1-qt_2}.
    \]
    Hence, $\gamma$ is a U-good class, and:
    \[
    \mathcal{B}_y(\gamma_{{\arrowvert}_{0}}) = \frac{1+yqt_1}{1-qt_1}, \quad \mathcal{B}_y(\gamma_{{\arrowvert}_{\infty}}) = \frac{1+yqt_2}{1-qt_2}, \quad \mathcal{B}_y(\gamma) = \frac{1 +yq\mathcal{O}(-1)}{1-q\mathcal{O}(-1)}.
    \]
\end{example}
The balancing of operators defined in \Cref{operator for operators} and $K$-theoretic functions is compatible in some cases as shown in the following theorem.
       \begin{thm}
    Let $\mathcal{I} = \sum_{d \geq 0}\frac{1}{\prod_{i=1}^n(qa_i)_d}Q^d$; here $a_i \in K_T(X)$, and for each $1 \leq i \leq n$, $a_i|_{x} \in K_{T}(\mathrm{pt})$ is a monomial for every $x \in X^T$, with $\prod_{i=1}^n(1-a_i) = 0$ (see \Cref{Pochammer Symbols}). In the definition of $\mathcal{B}_y$, let $U$ be determined by the collection of one-dimensional $T \times \mathbb{C}^*_q$-modules with weights given by:
        \[
        \big\{q^l \chi: \text{$\chi$ is a weight of $a_{i}\arrowvert_x$ for $x \in X^T$, $1 \leq i \leq n$, and $l \in \mathbb{N}$ }\big\}.
        \] 
        Define the operator $\mathcal{D} =\prod_{i=1}^n (1-a_iq^{Q\partial _Q})-Q$. Then we have:
    \begin{itemize}
        \item $\mathcal{D} \mathcal{I} = 0$;
        \item $\mathcal{H}_y(\mathcal{D}) \mathcal{B}_y(\mathcal{I})=0$.
    \end{itemize}
\end{thm}
\begin{proof}
    We first prove $\mathcal{D} \mathcal{I} = 0$. 
    \[
    \left(\prod_{i=1}^n (1-a_iq^{Q\partial _Q})\right)\mathcal{I} = \prod_{i=1}^n(1-a_i) + \sum_{d \geq 1}\frac{Q^d}{\prod_{i=1}^{n}(qa_i)_{d-1}} = Q\sum_{d \geq 1}\frac{Q^{d-1}}{\prod_{i=1}^{n}(qa_i)_{d-1}} = Q\mathcal{I},
    \]
    which gives us:
    \[
    \left(\prod_{i=1}^n (1-a_iq^{Q\partial _Q})-Q\right)\mathcal{I}=0. 
    \]
    Now we prove the second statement, namely $\mathcal{H}_y(\mathcal{D}) \mathcal{B}_y(\mathcal{I})=0$.
 First, by the definition of the balancing operator for classes, we obtain:
 \[
 \mathcal{B}_y(\mathcal{I}) = \sum_{d \geq 0}\frac{\prod_{i=1}^n(yqa_i)_d}{\prod_{i=1}^n(qa_i)_d}Q^d,
 \]
and by definition of balancing of operators we have:
\[
\mathcal{H}_y(\mathcal{D}) =\frac{\prod_{i=1}^n (1-a_iq^{Q\partial _Q})}{\prod_{i=1}^n(1-ya_iq^{Q\partial _Q})}-Q.
\]
The statement we want to prove is equivalent to:
\begin{equation}\label{Compatible balancing}
\left(\prod_{i=1}^n (1-a_iq^{Q\partial _Q})\right)\mathcal{B}_y(\mathcal{I}) = \left(\prod_{i=1}^n (1-ya_iq^{Q\partial})\right) Q\mathcal{B}_y(\mathcal{I})
\end{equation}
We first calculate the left hand side of \eqref{Compatible balancing}:
\[
\left(\prod_{i=1}^n (1-a_iq^{Q\partial _Q})\right)\mathcal{B}_y(\mathcal{I}) = \prod_{i=1}^n(1-a_i) + \sum_{d \geq 0}\frac{\prod_{i=1}^n(yqa_i)_d}{\prod_{i=1}^n(qa_i)_{d-1}}Q^d.
\]
The right hand side of \eqref{Compatible balancing} is:
\[
\left(\prod_{i=1}^n (1-ya_iq^{Q\partial})\right) Q\mathcal{B}_y(\mathcal{I}) = \sum_{d \geq 1}\frac{\prod_{i=1}^n(yqa_i)_d}{\prod_{i=1}^n(qa_i)_{d-1}}Q^d,
\]
and since $\prod_{i=1}^n(1-a_i) = 0$, we conclude the proof of the theorem.
\end{proof}
       \begin{example}\label{Balancing of operators for projective space}
           For projective spaces $G(1,n) \cong \mathbb{P}^{n-1}$, we encounter difference operators that annihilate the  $I$-functions $I(t;q,Q)$.  These operators take the form:
        \begin{equation} \label{Mojaz}
        \prod_{i=1}^{n}(1-\frac{Pq^{Q\partial_Q}}{t_i}) - Q.
        \end{equation}
        The $I$-function of $\mathbb{P}^{n-1}$ is also known:
        \[
 I(t;q,Q)=\sum_{d \geq 0}I_d(t;q)Q^d=\sum_{d \geq 0}\frac{Q^d}{(qPt_1^{-1})_d\cdot(qPt_2^{-1})_d}.
 \]
 We have the following:
        \[
        \left(\prod_{i=1}^{n}(1-\frac{Pq^{Q\partial_Q}}{t_i}) - Q\right)I(t;q,Q) = 0,
        \]
        where $P = \mathcal{O}(-1)$ is the tautological bundle on $\mathbb{P}^{n-1}$, and the ring $R$ defined in \cref{operator for operators} is $\mathbb{C}[P,t_1^{-1},\dots,t_n^{-1},Q]$.
Now, we can apply the operator $\mathcal{H}_{\hbar}$ \eqref{H-operator} to the operator \eqref{Mojaz}, and obtain:
        \[
        \mathcal{H}_{\hbar}\left(\prod_{i=1}^n(1-\frac{Pq^{Q\partial_Q}}{t_i})-Q\right) = \prod_{i=1}^n\frac{(1-\frac{Pq^{Q\partial_Q}}{t_i})}{(1-\hbar\frac{Pq^{Q\partial_Q}}{t_i})}-Q.
        \]
        As introduced in the previous section, these balanced $Q$-difference operators annihilate $\mathcal{B}_y\left(I(t;q,Q)\right)$:
        \[
        \left(\prod_{i=1}^{n}\frac{(1-\frac{Pq^{Q\partial_Q}}{t_i})}{(1-\hbar\frac{Pq^{Q\partial_Q}}{t_i})}-Q\right)\mathcal{B}_y\left(I(t;q,Q)\right)=0,
        \]
        or equivalently:
        \[
        \left(\prod_{i=1}^{n}(1-\frac{Pq^{Q\partial_Q}}{t_i}) \right)\mathcal{B}_y\left(I(t;q,Q)\right) = \left(\prod_{i=1}^{n}(1-\hbar\frac{Pq^{Q\partial_Q}}{t_i})\right)Q\mathcal{B}_y\left(I(t;q,Q)\right).
        \]
       \end{example}
        This example shows that the balancing of the $I$-functions is compatible with the balancing of the operators that annihilate them in the case of projective spaces.

Later, in \Cref{Main results}, after introducing the $J$-function and the quasimap vertex function, we show that each coefficient in the quasimap vertex function of $T^*G(r,n)$, at every fixed point, is the  $\lambda_y$-balanced $K$-theoretic class of the corresponding coefficient of the
$J$-function of $G(r,n)$, restricted to the same fixed point. This holds for a natural choice of $U$.

\subsection{Bethe-Ansatz Equations}\label{Bethe section}
In this section, we recall the Bethe-Ansatz equations for $T^*G(r,n)$, which arise in the study of integrable systems. We also discuss the abelian/non-abelian correspondence in quantum $K$-theory, as well as the operators that annihilate a special function, $\overline{J}^{tw}$, in the twisted quantum $K$-theory of the abelianized space of $G(r,n)$. We then construct new operators that annihilate $\mathcal{B}_y(\overline{J}^{tw})$. These operators recover the Bethe-Ansatz equations of $T^*G(r,n)$ after an appropriate specialization of parameters. 

The Grassmannian $G(r,n)$, viewed as a GIT quotient, is $Hom(\mathbb{C}^r ,\mathbb{C}^n)//_{\det} GL_r$. Its abelianization is defined to be the GIT quotient $Hom(\mathbb{C}^r , \mathbb{C}^n)//_{\det} {(\mathbb{C}^*)}^r$, with each factor of $\mathbb{C}^*$ scaling the columns of the $n \times r$ matrix. Therefore, the abelianization is isomorphic to $(\mathbb{P}^{n-1})^r$. In \cite{MR2367174}, it is proved that there exists a surjective map $\phi:(K(\mathbb{P}^{n-1}))^W \to K(G(r,n))$, where $W \cong S_r$ denotes the Weyl group of $GL_r$, acting by permutations on the factors. This map can be extended to the quantum $K$-rings $K(\mathbb{P}^{n-1})[[Q_1,\dots,Q_r]][\lambda]^W \to K(G(r,n))[[Q]]$, where, on $(\mathbb{P}^{n-1})^r$, we consider a twisted quantum $K$-theory. See \cite{huqkuruvilla2025relationstwistedquantumkrings}. By slight abuse of notation, we denote this map again by $\phi$. Under $\phi$, all the quantum parameters $Q_1,\dots,Q_r$ map to the quantum parameter $Q$, and $\lambda$ maps to 1. 

Givental and Yan in \cite{Givental:2020} calculate the following function in the twisted theory of $(\mathbb{P}^{n-1})^r$:
\begin{equation*}\label{twisted function}
\overline{J}^{tw}(q):=\sum_{d_1,\dots,d_n\geq 0}\frac{\prod_iQ_i^{d_i}}{\prod_{i=1}^n\prod_{j=1}^{n}\prod_{m=1}^{d_i}(1-q^m\frac{P_i}{t_j})}\prod_{i\neq j}\frac{\prod_{m=-\infty}^{d_i-d_j} (1-q^m\lambda\frac{P_i}{P_j})}{\prod_{m=-\infty}^{0} (1-q^m\lambda\frac{P_i}{P_j})},
\end{equation*}
where each $P_i$ is the tautological bundle of the $i$-th projective space.
This function maps to the $J$-function of $G(r,n)$ under $\phi$. Therefore, it can be used to find the relations in $QK^{tw}((\mathbb{P}^{n-1})^r)$; see \cite{huqkuruvilla2025relationstwistedquantumkrings}.

$\overline{J}^{tw}$ satisfies the following $q-$difference equations for $1 \leq i \leq r$:

$$\prod_{j\neq i} (1-\lambda qP_jP_i^{-1}q^{Q_j\partial_{Q_j}-Q_i\partial_{Q_i}})\prod_a(1-P_iq^{Q_i\partial_{Q_i}}/\Lambda_a)\overline{J}^{tw}=Q_i\prod_{j\neq i}(1-\lambda qP_iP_j^{-1}q^{Q_i\partial_{Q_i}-Q_j\partial_{Q_j}})\overline{J}^{tw}.$$
By extending the definition of the balancing (see \Cref{Appendix A}), we can define $\mathcal{J}$ to be the $\lambda_y$-balanced $K$-theoretic class of $\overline{J}^{tw}$, that is, $\mathcal{J} := \mathcal{B}_y(\overline{J}^{tw})$. Explicitly we have:

$$\mathcal{J}=\sum_{d_1,\dots,d_n\geq 0}\prod_{i=1}^rQ_i^{d_i}\frac{\prod_{i=1}^n\prod_{j=1}^{n}\prod_{m=1}^{d_i}(1+yq^m\frac{P_i}{t_j})}{\prod_{i=1}^n\prod_{j=1}^{n}\prod_{m=1}^{d_i}(1-q^m\frac{P_i}{t_j})}\prod_{i\neq j}\frac{\prod_{m=-\infty}^{d_i-d_j} (1-q^m\lambda\frac{P_i}{P_j})\prod_{m=-\infty}^{0} (1+yq^m\lambda\frac{P_i}{P_j})}{\prod_{m=-\infty}^{0} (1-q^m\lambda\frac{P_i}{P_j}) \prod_{m=-\infty}^{d_i-d_j} (1+yq^m\lambda\frac{P_i}{P_j})}.$$

For $\mathcal{J}$, we have the following $q$-difference equation for $1 \leq i \leq r$ (see \Cref{Appendix B} for the proof):
\begin{equation}\label{q-difference eq}
\mathcal{D}_1^i\mathcal{J} = \mathcal{D}_2^i\mathcal{J},
\end{equation}
where:
\[
\mathcal{D}_1^i = \prod_{j\neq i}(1+y\lambda P_iP_j^{-1}q^{Q_i\partial_{Q_i}-Q_j\partial_{Q_j}})\prod_{j\neq i} (1-\lambda qP_jP_i^{-1}q^{Q_j\partial_{Q_j}-Q_i\partial_{Q_i}})\prod_a(1-P_iq^{Q_i\partial_{Q_i}}/t_a),
\]
and
\[
\mathcal{D}_2^i = \prod_{j\neq i} (1+y\lambda qP_jP_i^{-1}q^{Q_j\partial_{Q_j}-Q_i\partial_{Q_i}})\prod_a(1-P_iq^{Q_i\partial_{Q_i}}/t_a)Q_i\prod_{j\neq i}(1-\lambda qP_iP_j^{-1}q^{Q_i\partial_{Q_i}-Q_j\partial_{Q_j}}).
\]
Consequently, the operators $f_i(y,\lambda,q,Q_1,\dots,Q_r,P_1q^{Q_1\partial_{Q_1}},\dots, P_rq^{Q_r\partial_{Q_r}})=\mathcal{D}_1^i - \mathcal{D}_2^i$ annihilate $\mathcal{J}$ for $1 \leq i \leq r$: 
\[(\mathcal{D}_1^i - \mathcal{D}_2^i) \mathcal{J} = 0. \]
Define $q$-difference equations $g_i$ via the change of variables $P_iq^{Q_i\partial_{Q_i}} \mapsto P_i $, $q \mapsto 1$, and $y \mapsto -\hbar$:
\[
g_i:=f_i(-\hbar,\lambda,1,Q_1,\dots,Q_r,P_1,\dots,P_r) = 0.
\]
We can rewrite $g_i=0$ as follows:
\begin{align}\label{IMT-type equation}
\prod_{j\neq i}(1-\hbar\lambda P_iP_j^{-1})\prod_{j\neq i} (1-\lambda P_jP_i^{-1})\prod_a(1-P_i/t_a) - {\prod_{j\neq i} (1-\hbar\lambda P_jP_i^{-1})\prod_a(1-P_i/t_a)Q_i\prod_{j\neq i}(1-\lambda P_iP_j^{-1})}=0.
\end{align} 
Now we recall the Bethe-Ansatz equations, which are fundamental in the study of integrable systems. The Bethe-Ansatz equations for $T^*G(r,n)$ are given by:
\begin{equation}\label{Bethe}
    \prod_{j=1}^n\frac{x_i/t_j-1}{kx_i/t_j-s}+(-1)^rQ\prod_{j\neq i}^r\frac{kx_j-sx_i}{kx_i-sx_j}=0,
    \qquad i=1,\dots,r
\end{equation}
where we should set $k=-y$ and $s=1$ in order to match with the geometric left Weyl group action. Furthermore, we also consider the change of variables $y = -\hbar$.

If we extend the map $\phi$ by setting $\phi(\hbar) = \hbar$, we obtain the following:

\begin{prop}
    Under the abelian/non-abelian map $\phi$ and after the change of variables $P_i \mapsto x_i$, equations \eqref{IMT-type equation} map to the Bethe-Ansatz equations on $T^*G(r,n)$.
\end{prop}
\begin{proof}
    In order to find the image of \Cref{IMT-type equation} under $\phi$, one needs to specialize $\lambda \mapsto 1$ and $Q_i \to Q$ for $1 \leq i \leq r$:
    \[
\prod_{j \ne i}(1- \hbar  \frac{P_i}{P_j})\prod_{j \ne i}(1-\frac{P_j}{P_i})\prod_{a}(1-\frac{P_i}{t_a}) - Q \prod_{j \ne i}(1-\hbar  \frac{P_j}{P_i})\prod_{a}(1-\hbar \frac{P_i}{t_a})\prod_{j \ne i}(1-\frac{P_i}{P_j})=0.
\]
Rearranging terms, this expression becomes:
\[
\frac{\prod_{j \ne i}(1- \hbar  \frac{P_i}{P_j})\prod_{j \ne i}(1-\frac{P_j}{P_i})\prod_{a}(1-\frac{P_i}{t_a})}{\prod_{j \ne i}(1-\hbar  \frac{P_j}{P_i})\prod_{a}(1-\hbar \frac{P_i}{t_a})\prod_{j \ne i}(1-\frac{P_i}{P_j})}-Q=0.
\]
 Note that since $\frac{1-\frac{P_j}{P_i}}{1-\frac{P_i}{P_j}} = -\frac{P_j}{P_i}$, we obtain:

 \[
 (-1)^{r-1} \prod_{j \ne i} \frac{P_j}{P_i}\cdot\frac{\prod_{a}(1-\frac{P_i}{t_a})\prod_{j \ne i}(1- \hbar  \frac{P_i}{P_j})}{\prod_{a}(1-\hbar \frac{P_i}{t_a})\prod_{j \ne i}(1-\hbar  \frac{P_j}{P_i})}-Q=0,
 \]
 which equals:
 \[
  \prod_{j \ne i} \frac{P_j}{P_i}\cdot\frac{\prod_{a}(1-\frac{P_i}{t_a})\prod_{j \ne i}(1- \hbar  \frac{P_i}{P_j})}{\prod_{a}(1-\hbar \frac{P_i}{t_a})\prod_{j \ne i}(1-\hbar  \frac{P_j}{P_i})}+ (-1)^rQ=0,
 \]
 which simplifies to: 
 \[
 \prod_a \frac{1-\frac{P_i}{t_a}}{1-\hbar\frac{P_i}{t_a}}+(-1)^r Q\prod_{j \ne i}^r \frac{\hbar P_j - P_i}{\hbar P_i - P_j}. 
 \]
 Finally, after the change of variables $P_i \mapsto x_i$, we recover the Bethe-Ansatz equations \eqref{Bethe}. This completes the proof.
\end{proof}
 
\section{Quot scheme}\label{sec:Quot scheme} In this section, we recall the definition of the
Grothendieck Quot scheme, following \cite{MR908717}. 

Let $f:\mathbb{P}^1 \to G(r,n)$ be a rational curve of degree $d$ in the Grassmannian. 
By the universal property of $G(r,n)$, such a map $f$ corresponds to an
isomorphism class of short exact sequences of locally free
sheaves:
\[
0 \to \mathcal{K} \to \mathcal{O}^n_{\mathbb{P}^1} \to \mathcal{Q} \to 0,
\]
where $\operatorname{rank}(\mathcal{K}) = r$,
 and $\deg(\mathcal{K})=-d$. Since $\mathcal{K}$ is a locally free sheaf on
$\mathbb{P}^1$, it follows that:
\[
\mathcal{K} \simeq \bigoplus_{i=1}^r \mathcal{O}_{\mathbb{P}^1}(-d_i)
\]
with $d_i \ge 0$ and $\sum_{i=1}^r d_i = d$.

The quotients $\mathcal{Q}$ fitting into short exact sequences of the form 
\[
0 \to \bigoplus_{i=1}^r \mathcal{O}_{\mathbb{P}^1}(-d_i) \to \mathcal{O}^n_{\mathbb{P}^1} \to \mathcal{Q} \to 0 
\]
may degenerate in flat families to quotients $\mathcal{Q}$ that are no longer locally free. Consequently, the moduli space of rational curves in $G(r,n)$ is non-compact.

The Grothendieck Quot scheme, denoted by $Quot_{\mathbb{P}^1,d}(\mathbb{C}^n,r)$, is a variety that compactifies this space. It parametrizes coherent quotient sheaves of the form: 
\[
\mathcal{O}_{\mathbb{P}^1}^n \to \mathcal{Q} \to 0,
\]
with Hilbert polynomial $p(t) = (n-r)(t+1)+d$. Its
closed points correspond to isomorphism classes of short exact sequences of the form:
\[
0 \to \bigoplus_{i=1}^r \mathcal{O}_{\mathbb{P}^1}(-d_i) \to \mathcal{O}^n_{\mathbb{P}^1} \to \mathcal{Q} \to 0
\]
where $\mathcal{Q}$ is a coherent sheaf with Hilbert polynomial $p(t) = (n-r)(t+1)+d$, and the integers $d_i \geq 0$ satisfy $\sum_{i=1}^r d_i = d$.

The Quot scheme may also be defined more precisely as a functor of points; we refer the reader to \cite{GRo} for this perspective. In this paper, however, we follow the approach of Strømme \cite{MR908717}, who also proved several foundational results on Quot schemes over Grassmannians. The key property of the Quot scheme utilized in this work is that it is a smooth projective variety. In the next section, we recall the torus action defined by Strømme on the Quot scheme and show that the set of torus-fixed points is finite. See \cite[Theorem 2.1]{MR908717}.

\subsection{Torus Action}\label{Torussection} In this section, we recall the torus actions on the Quot scheme, as described in \cite{Bertram:2003qd, MR908717}, which we will use later to define a $\lambda_y$-balanced K-theoretic class on the Grassmannian $G(r,n)$.

We begin by recalling a $\mathbb{C}^*$-action on the Quot scheme, and the description of $\mathbb{C}^*$-fixed points, as described in \cite{Bertram:2003qd}. 
Recall that points in the Quot scheme can be identified with rank $r$, locally free subsheaves $\mathcal{K}$ of $\mathbb{C}^n \otimes \mathcal{O}_{\mathbb{P}^1}$ of degree $-d$.
There is a natural $\mathbb{C}^*$-action on $\mathbb{P}^1$ given by
\begin{align} \label{action}
    q\cdot[x:y] \longmapsto [q x:y],
\end{align}
which we denote by $\mathbb{C}_q^*$. This action lifts naturally to the trivial bundle $\mathbb{C}^n \otimes \mathcal{O}_{\mathbb{P}^1}$, and thus induces an action on $Quot_{\mathbb{P}^1,d}(\mathbb{C}^n,r)$ by pulling back kernels.

To study the $\mathbb{C}_q^*$-fixed points of $Quot_{\mathbb{P}^1,d}(\mathbb{C}^n,r)$, one needs to study $Quot_{\mathbb{P}^1,d}(\mathbb{C}^r,r)$, the moduli space parameterizing degree $d$ torsion quotient sheaves of 
$\mathbb{C}^r \otimes \mathcal{O}_{\mathbb{P}^1}$. Taking the top exterior power defines a natural map:
\[
\wedge^r : Quot_{\mathbb{P}^1,d}(\mathbb{C}^r,r) \to Quot_{\mathbb{P}^1,d}(\mathbb{C},r)
\]
given by
\[
\mathcal{K} \subset \mathbb{C}^r \otimes \mathcal{O}_{\mathbb{P}^1} \mapsto \wedge^r \mathcal{K} \subset \wedge^r (\mathbb{C}^r \otimes \mathcal{O}_{\mathbb{P}^1}) \cong \mathbb{C} \otimes \mathcal{O}_{\mathbb{P}^1}
\]

As a locally free sheaf on $\mathbb{P}^1$, $\mathcal{K}$ splits as a direct sum of line bundles:
\[
\mathcal{O}_{\mathbb{P}^1}(-d_1) \oplus \dots \oplus \mathcal{O}_{\mathbb{P}^1}(-d_r), 
\]
where $d_i \ge 0$ and $d_1+\dots+d_r=d$. This splitting is unique if we require $0 \le d_1 \le \dots \le d_r$.

 We now state the following lemma from \cite[Lemma 1.1]{Bertram:2003qd}:
\begin{lemma}\label{lemma 3.1}
    For each splitting type $\{d_i\}$ as above, let $m_1 <m_2 <\dots<m_k$ denote the jumping indices (i.e., $0 \leq d_1=\dots=d_{m_1} < d_{m_1 +1}=\dots=d_{m_2}<\dots$). Then there is an embedding of the flag manifold:
    \[
    i_{\{d_i\}}:\mathrm{Fl}(m_1,m_2,\dots,m_k;r) \hookrightarrow Quot_{\mathbb{P}^1,d}(\mathbb{C}^r,r)
    \]
    with the property that each fixed point of the $\mathbb{C}_q^*$-action with $\mathrm{supp}(\mathbb{C}^n \otimes \mathcal{O}_{\mathbb{P}^1})/\mathcal{K})=\{0\}$ and with kernel splitting type $\{d_i\}$ corresponds to a point of the image of $i_{\{d_i\}}$.        
\end{lemma}
 If a scheme $X$ is equipped with a vector bundle $E$ of rank $r$, then there exists a relative Quot scheme over $X$, 
\[
Quot_{\mathbb{P}^1,d}(E,r) \to X,
\]
which represents the functor assigning to each $X$-scheme $T$ the set of   flat, relative length $d$ quotients $\pi^*E \twoheadrightarrow \mathcal{Q}$ over $T$ for $\pi: \mathbb{P}^1 \times T \to T \to X$. Grothendieck showed that the fibers of this relative Quot scheme are isomorphic to $Quot_{\mathbb{P}^1,d}(\mathbb{C}^r,r)$. Moreover,  the $\mathbb{C}_q^*$-action globalizes, and we obtain a morphism of $X$ schemes:
\[
i_{\{d_i\}}:\mathrm{Fl}(m_1,\dots,m_k;E) \to Quot_{\mathbb{P}^1,d}(E,r).
\]
By applying the above construction to the universal sub-bundle $S$ on $G(r,n)$ we obtain the following lemma   (\cite[Lemma 1.2]{Bertram:2003qd}):

\begin{lemma}\label{Relative fixed-points}
    There is a natural $\mathbb{C}_q^*$-equivariant embedding
    \[
    j:Quot_{\mathbb{P}^1,d}(S,r) \hookrightarrow Quot_{\mathbb{P}^1,d}(\mathbb{C}^n,r) 
    \]
    such that all the $\mathbb{C}_q^*$-fixed points of $Quot_d(\mathbb{C}^n,r)$ that satisfy $\mathrm{supp}(\mathrm{tor}(\mathbb{C}^n \otimes \mathcal{O}_{\mathbb{P}^1})/\mathcal{K}))=\{0\}$ are precisely the images of flag manifolds
    \[
    i_{\{d_i\}}: \mathrm{Fl}(m_1 , m_2, \dots , m_k,r:n)=\mathrm{Fl}(m_1 , m_2 , \dots , m_k;S) \to Quot_{\mathbb{P}^1,d}(S,r)
    \]
    embedded by the relative version of \Cref{lemma 3.1}. In other words:
    \[
    Quot_{\mathbb{P}^1,d}(\mathbb{C}^n,r)^{\mathbb{C}^*}_0 = \sqcup_{\{d_i\}}i_{\{d_i\}} (\mathrm{Fl}(m_1,\dots,m_k;S)),
    \]
    where $Quot_{\mathbb{P}^1,d}(\mathbb{C}^n,r)^{\mathbb{C}_q^*}_0$ are $\mathbb{C}_q^*$-fixed points supported at 0.
\end{lemma}
Note that on closed points the embedding $j$ is given by:
\[
(\mathcal{K} \hookrightarrow \mathcal{S}_x \otimes \mathcal{O}_{\mathbb{P}^1}) \mapsto (\mathcal{K} \hookrightarrow \mathbb{C}^n \otimes \mathcal{O}_{\mathbb{P}^1})
\]
where $x \in G(r,n)$, $\mathcal{S}_x$ denotes the fiber of $\mathcal{S}$ over $x$, and $\mathcal{K} \hookrightarrow \mathbb{C}^n \otimes \mathcal{O}_{\mathbb{P}^1}$ is a composition 
$\mathcal{K} \hookrightarrow \mathcal{S}_x \otimes \mathcal{O}_{\mathbb{P}^1} \to \mathbb{C}^n \otimes \mathcal{O}_{\mathbb{P}^1} $. Here, the last morphism is induced by the inclusion map $ \mathcal{S}_x\hookrightarrow \mathbb{C}^n$ on $G(r,n)$.

We now recall from \cite{MR908717} the $T \times \mathbb{C}_q^*$-action on the Quot scheme, where $T \subset GL_n(\mathbb{C})$ denotes the maximal torus of diagonal matrices. 
First, observe that the torus $T$ acts on $\mathbb{C}^n$ by scaling each coordinate: 
\begin{equation} \label{Torus action}
\mathrm{diag}(t_1,\dots,t_n)\cdot(x_1,\dots,x_n) = (t_1x_1,\dots,t_nx_n).
\end{equation}
This induces a natural $T$-action on the Quot scheme $Quot_{\mathbb{P}^1,d}(\mathbb{C}^n,r)$ as follows. Consider a point of the Quot scheme represented by the short exact sequence:
\[
0 \to \mathcal{K} \to \mathbb{C}^n \otimes \mathcal{O}_{\mathbb{P}^1}. 
\]

The above action of $T$ on $\mathbb{C}^n$ induces an automorphism of the trivial bundle  $\mathbb{C}^n \otimes \mathcal{O}_{\mathbb{P}^1}$. 
The action of $t \in T$ on the point defined by the above sequence is given by applying this automorphism to the subbundle $\mathcal{K}$:
\[
t\cdot(0 \to \mathcal{K} \to \mathbb{C}^n \otimes \mathcal{O}_{\mathbb{P}^1}) := 0 \to t(\mathcal{K}) \to t(\mathbb{C}^n \otimes \mathcal{O}_{\mathbb{P}^1}) \cong \mathbb{C}^n \otimes \mathcal{O}_{\mathbb{P}^1}.
\]

We now analyze the fixed points of the $T \times \mathbb{C}_q^*$ action. We begin by describing the fixed points under the action of $T$. Since $T$ is a torus, the $T$-module $\mathbb{C}^n$ decomposes into one-dimensional weight spaces:
\[
\mathbb{C}^n = \mathbb{C}_{t_1} \oplus \dots \oplus \mathbb{C}_{t_n},
\]
 where each $\mathbb{C}_{t_i}$ is a one-dimensional $T$-module corresponding to the character $t_i :T \to \mathbb{C}^*$, defined by $t \mapsto t_i$.

A point
\[
0 \to \mathcal{K} \to \mathbb{C}^n \otimes \mathcal{O}_{\mathbb{P}^1}
\]
 is fixed under the $T$-action if and only if $\mathcal{K}$ decomposes as $\mathcal{K} = \bigoplus_{i=1}^r \mathcal{K}_i$, where: 
 \[
 \mathcal{K}_i = \mathcal{K} \cap (\mathbb{C}_{t_i} \otimes \mathcal{O}_{\mathbb{P}^1}).
 \]
  Since $\mathbb{C}_{t_i} \otimes \mathcal{O}_{\mathbb{P}^1}$ is a trivial line bundle, and  $\mathcal{K}_i$ is a subsheaf of it, $\mathcal{K}_i$ must be isomorphic to $\mathbb{C}_{t_i} \otimes \mathcal{O}_{\mathbb{P}^1}(-d_i)$ for some $d_i \geq 0$. The inclusion
\[
\mathbb{C}_{t_i} \otimes \mathcal{O}_{\mathbb{P}^1}(-d_i) \to \mathbb{C}_{t_i} \otimes \mathcal{O}_{\mathbb{P}^1}
\]
is then  given by multiplication by a section $f_i \in H^0(\mathbb{P}^1,\mathcal{O}_{\mathbb{P}^1}(d_i))$, i.e., a homogeneous polynomial of degree $d_i$. 

 A $T$-fixed point
 \[
 0 \to \mathcal{K} \to \mathbb{C}^n \otimes \mathcal{O}_{\mathbb{P}^1}
 \]
is also fixed under the $\mathbb{C}_q^*$-action if each $f_i$ is a monomial in $x$ and $y$; that is, $f_i=x^{a_i}y^{b_i}$ with $a_i+b_i=d_i$. Thus, a general $T \times \mathbb{C}_q^*$-fixed point of the Quot scheme is of the form
\begin{equation}\label{1}
 \bigoplus_{i=1}^n \delta_i\mathbb{C}_{t_i} \otimes \mathcal{O}_{\mathbb{P}^1}(-a_i-b_i) \to \mathbb{C}^n \otimes \mathcal{O}_{\mathbb{P}^1},
\end{equation}
subject to the following constraints:
\begin{align*}
\delta_i \in \{0,1\} , \forall i, \\
    \sum \delta_i = r, \\
    \sum a_i+b_i = d,
\end{align*}
where the map $\mathbb{C}_{t_i} \otimes \mathcal{O}_{\mathbb{P}^1}(-a_i-b_i) \to \mathbb{C}_{t_i} \otimes \mathcal{O}_{\mathbb{P}^1}$ is given by the monomial $x^{a_i}y^{b_i}$, for all $i$ such that $\delta_i \ne 0$.

It follows that there is a combinatorial description of $T \times \mathbb{C}^*$-fixed points of the Quot scheme. In fact, the discussion above shows that there is a one-to-one correspondence between the fixed points of the Quot scheme and the subset of $\mathbb{Z}^{3n}$ given by: 
\[
\left\{
(a_1,\dots,a_n,b_1,\dots,b_n,\delta_1,\dots,\delta_n) \in \mathbb{Z}^{3n}
\ \Bigg|\
\begin{array}{l}
\delta_i \in \{0,1\}\\
\sum \delta_i = r\\
\sum (a_i + b_i) = d
\end{array}
\right\}.
\]
We denote the set of $T \times \mathbb{C}_q^*$-fixed point of the Quot scheme by $Quot_{\mathbb{P}^1,d}(\mathbb{C}^n,r)^{T \times \mathbb{C}_q^*}$.

A $T \times \mathbb{C}_q^*$-fixed point of the form
\[
0 \to \bigoplus \delta_i\mathbb{C}_{t_i} \otimes \mathcal{O}_{\mathbb{P}^1}(-a_i-b_i) \to \mathbb{C}^n \otimes \mathcal{O}_{\mathbb{P}^1}
\]
 is supported at zero if and only if all $b_i$=0. In other words, such a fixed point takes the form:
\begin{equation} \label{E:quotfixedpt}
0 \to \bigoplus \delta_i\mathbb{C}_{t_i} \otimes \mathcal{O}_{\mathbb{P}^1}(-d_i) \to \mathbb{C}^n \otimes \mathcal{O}_{\mathbb{P}^1}
\end{equation}
where each inclusion 
\[
\mathbb{C}_{t_i} \otimes \mathcal{O}_{\mathbb{P}^1}(-d_i) \to \mathbb{C}_{t_i} \otimes \mathcal{O}_{\mathbb{P}^1}
\]
 is given by multiplication by $x^{d_i}$ for all $i$ such that $\delta_i \ne 0$.
 We denote the set of $T \times \mathbb{C}_q^*$-fixed point of the Quot scheme supported at zero by $Quot_{\mathbb{P}^1,d}(\mathbb{C}^n,r)_{0}^{T \times \mathbb{C}_q^*}$.

 We now analyze the tangent space of the Quot scheme at  $T \times \mathbb{C}_q^*$-fixed points that are supported at zero, and explicitly describe its decomposition into one-dimensional $T \times \mathbb{C}_q^*$-modules.

 \begin{lemma}\label{fixed-point lemma}
Assume that $x$ is a $T \times \mathbb{C}_q^*$-fixed point of the Quot scheme corresponding to $\bigoplus_i \delta_i\mathbb{C}_{t_i} \otimes \mathcal{O}_{\mathbb{P}^1}(-a_i-b_i) \to \mathbb{C}^n \otimes \mathcal{O}_{\mathbb{P}^1}$. 
Then the $T \times \mathbb{C}_q^*$-module $T_xQuot_{\mathbb{P}^1,d}(\mathbb{C}^n,r)$ decomposes into a direct sum of one-dimensional weight spaces with the following weights:
\begin{itemize}
    \item $\frac{t_j}{t_i}.q^{-a_i+r}$ for $(i,j)$ such that $\delta_i= 1 , \delta_j=0$ and $0 \leq r \leq a_i+b_i$;
    \item $q^{-a_i+r}$  for $(i,j)$ such that $i=j$,$\delta_i=1$ and $0 \leq r \leq a_i+b_i$ and $r \neq a_i$;
    \item $\frac{t_j}{t_i}.q^r$ for $(i,j)$ such that $i \ne j, \delta_i = \delta_j=1$ and $-a_i \leq r \leq -a_i+a_j-1$, $a_j \geq 1$;
    \item $\frac{t_j}{t_i}.q^r$ for $(i,j)$ such that $i \ne j, \delta_i = \delta_j=1$ and $b_i-b_j+1 \leq r \leq b_i$, $b_j \geq 1.$
\end{itemize}
 \end{lemma}

 \begin{proof}
 We prove this lemma only for fixed points of the Quot scheme that are supported at zero, as this is the case relevant to our paper. Therefore, we restrict our attention to points corresponding to: 
 \[
 \bigoplus_i \delta_i\mathbb{C}_{t_i} \otimes \mathcal{O}_{\mathbb{P}^1}(-d_i) \to \mathbb{C}^n \otimes \mathcal{O}_{\mathbb{P}^1},
 \]
 where $a_i = d_i$, and $b_i=0$. In this situation, the fourth type of weight listed in the lemma does not contribute to the decomposition.
 
    By Grothendieck's theorem \cite[cor 5.3]{GRo} the tangent space of the Quot scheme at a point $0 \to \mathcal{A} \to \mathbb{C}^n \otimes \mathcal{O}_{\mathbb{P}^1} \to \mathcal{B} \to 0$ is isomorphic to the vector space
    \[
    Hom(\mathcal{A},\mathcal{B}).
    \]

Therefore, the tangent space at the fixed point $x$ is given by:
\[
Hom(\bigoplus_i \delta_i\mathbb{C}_{t_i} \otimes \mathcal{O}_{\mathbb{P}^1}(-d_i),\mathbb{C}^n \otimes \mathcal{O}_{\mathbb{P}^1}/\oplus_i \delta_i\mathbb{C}_{t_i} \otimes \mathcal{O}_{\mathbb{P}^1}(-d_i)). 
\]
This can be written as:
\begin{equation} \label{Tangent Decomposition}
    \bigoplus_{i,j} Hom(\delta_i\mathbb{C}_{t_i} \otimes \mathcal{O}_{\mathbb{P}^1}(-d_i),\mathbb{C}_{t_j} \otimes\mathcal{O}_{\mathbb{P}^1} /\delta_j.\mathbb{C}_{t_j} \otimes \mathcal{O}_{\mathbb{P}^1}(-d_j)).
\end{equation}
We now consider three cases for the sum in the above expression:\\

\begin{itemize}

\item Case 1: $\delta_i= 1 , \delta_j=0$.\\
    In this case, the corresponding component is:
    \[
    Hom(\mathbb{C}_{t_i} \otimes \mathcal{O}_{\mathbb{P}^1}(-d_i),\mathbb{C}_{t_j} \otimes \mathcal{O}_{\mathbb{P}^1}),
    \]
    which can be viewed as the inclusion of $\mathbb{C}[x,y]-$modules: 
    \[
    \langle x^{d_i} \rangle \hookrightarrow \mathbb{C}[x,y]. 
\]
    On the open set $U_0$ (where $x \neq 0$), the morphism is given by: 
    \[
    \langle1 \rangle  \hookrightarrow \mathbb{C}[\frac{y}{x}],
    \]
    which is represented by $p(\frac{y}{x})$ where $p$ is a polynomial in $\frac{y}{x}$. Similarly, on the open set $U_1$ (where $y \neq 0$), the morphism is: 
    \[
    \langle (\frac{x}{y})^{d_i} \rangle \to \mathbb{C}[\frac{x}{y}],
    \]
    which is given by $(\frac{y}{x})^{d_i}.q(\frac{x}{y})$ where $q(\frac{x}{y})$ is a polynomial in $\frac{x}{y}$. 
    Let $X=\frac{x}{y}$, and $Y = \frac{y}{x}$. On the intersection of $U_0$ and $U_1$, we have $X=\frac{1}{Y}$, and $(\frac{y}{x})^{d_i}q(\frac{x}{y}) = p(\frac{y}{x})$, so $p(\frac{1}{X}) = \frac{q(X)}{X^{d_i}}$. Thus, $p$ and also $q$ must both be polynomials of degree at most $d_i$. 
      It follows that on $U_0$, $Hom(\mathcal{O}_{\mathbb{P}^1}(-d_i) , \mathcal{O}_{\mathbb{P}^1})$ has a basis: 
     \[
     \{1,Y,\dots,Y^{d_i}\},
     \]
     and since $Y=\frac{y}{x}$, the weights are:
     \[
     1,q^{-1},\dots,q^{-d_i}.
     \]
     On $U_1$, $Hom(\mathcal{O}_{\mathbb{P}^1}(-d_i) , \mathcal{O}_{\mathbb{P}^1})$ has a basis:
     \[
     \frac{1}{X^{d_i}}\{1,X,\dots,X^{d_i}\} = \{1,\frac{1}{X},\dots,\frac{1}{X^{d_i}}\},
     \]
     and since $X=\frac{x}{y}$, the weights are:
     \[
     1,q^{-1},\dots,q^{-d_i}.
     \]
     Hence, these weights are weights of $Hom(\mathcal{O}_{\mathbb{P}^1}(-d_i) , \mathcal{O}_{\mathbb{P}^1})$.

Thus, the weights of $Hom(\mathbb{C}_{t_i} \otimes \mathcal{O}_{\mathbb{P}^1}(-d_i),\mathbb{C}_{t_j} \otimes \mathcal{O}_{\mathbb{P}^1})$ are:
\[
\frac{t_j}{t_i}.q^{-d_i+r},
\]
 where $0 \leq r \leq d_i$.

\item Case 2: $i=j$,$\delta_i=1$.\\
In this case, the corresponding component is 
\[
Hom(\mathbb{C}_{t_i} \otimes \mathcal{O}_{\mathbb{P}^1}(-d_i),\mathbb{C}_{t_i} \otimes \mathcal{O}_{\mathbb{P}^1}/\mathbb{C}_{t_i} \otimes \mathcal{O}_{\mathbb{P}^1}(-d_i)).
\]
Note that since $\mathbb{C}_{t_i}$ acts on both sides, the weight corresponding to $\mathbb{C}_{t_i}$ does not contribute, and we are therefore only tracking the $q$-weights here.

Consider the following short exact sequence: 
    \[0 \to \mathcal{O}_{\mathbb{P}^1}(-d_i) \to  \mathcal{O}_{\mathbb{P}^1} \to \mathcal{O}_{\mathbb{P}^1}/\mathcal{O}_{\mathbb{P}^1}(-d_i) \to 0.\]
    
    Applying the functor $Hom(\mathcal{O}_{\mathbb{P}^1}(-d_i),-)$, we obtain the exact sequence:
    \begin{align*}
   \scalebox{0.9}{$0 \to Hom(\mathcal{O}_{\mathbb{P}^1}(-d_i),\mathcal{O}_{\mathbb{P}^1}(-d_i)) \to  Hom(\mathcal{O}_{\mathbb{P}^1}(-d_i),\mathcal{O}_{\mathbb{P}^1}) \to Hom(\mathcal{O}_{\mathbb{P}^1}(-d_i), \mathcal{O}_{\mathbb{P}^1}/\mathcal{O}_{\mathbb{P}^1}(-d_i))   \to 0$}.
    \end{align*}
    As a vector space, we have the decomposition:
    \begin{align*}
    \scalebox{0.9}{$Hom(\mathcal{O}_{\mathbb{P}^1}(-d_i),\mathcal{O}_{\mathbb{P}^1}) = Hom(\mathcal{O}_{\mathbb{P}^1}(-d_i),\mathcal{O}_{\mathbb{P}^1}(-d_i)) \oplus Hom(\mathcal{O}_{\mathbb{P}^1}(-d_i), \mathcal{O}_{\mathbb{P}^1}/\mathcal{O}_{\mathbb{P}^1}(-d_i))$.}
    \end{align*}
    Note that
    \[
    Hom(\mathcal{O}_{\mathbb{P}^1}(-d_i),\mathcal{O}_{\mathbb{P}^1}(-d_i)) \cong H^0(\mathbb{P}^1,\mathcal{O}_{\mathbb{P}^1}) \cong \mathbb{C},
    \]
    which has weight 1. The weight decomposition of $Hom(\mathcal{O}_{\mathbb{P}^1}(-d_i),\mathcal{O}_{\mathbb{P}^1})$  is known from case 1, so the weights for $Hom(\mathcal{O}_{\mathbb{P}^1}(-d_i), \mathcal{O}_{\mathbb{P}^1}/\mathcal{O}_{\mathbb{P}^1}(-d_i))$ are given by $q^{-d_i+r}$, where $ 0 \leq r \leq d_i-1$. 

    \item Case 3: $i \ne j, \delta_i = \delta_j=1$.
    
     The computation proceeds analogously to Case 1. The only important point here is that if $d_j = a_j=0$, then in \eqref{Tangent Decomposition} we have terms:
     \[
     Hom(\mathbb{C}_{t_i} \otimes \mathcal{O}_{\mathbb{P}^1}(-d_i),\mathbb{C}_{t_j} \otimes\mathcal{O}_{\mathbb{P}^1} /\mathbb{C}_{t_j} \otimes \mathcal{O}_{\mathbb{P}^1}) = 0.
     \]
     Therefore, these terms contribute nothing, and we only need to focus of $a_j \geq 1$ for these terms.
    \end{itemize}
\end{proof}

\begin{Corollary}\label{Corollary2.5}
    Let $x$ be a $T \times \mathbb{C}_q^*$-fixed point supported at $0 \in \mathbb{P}^1$, given by:
    \[
    \mathbb{C}_{t_{i_1}}\mathcal{O}_{\mathbb{P}^1}(-d_{i_1}) \oplus \cdots \mathbb{C}_{t_{i_r}}\mathcal{O}_{\mathbb{P}^1}(-d_{i_r}) \to \bigoplus_{i=1}^n \mathbb{C}_{t_i} \mathcal{O}_{\mathbb{P}^1},
    \] 
    where $\{i_1, \cdots i_r\} \subset  \{1, \cdots n\}$. Then the following formula holds for the $\lambda_y$-class of the cotangent space of the Quot scheme at $x$:
    \[
    \lambda_y(T_x^*Quot) = ABC,
    \]
    where:
    \[
    A = \prod_{i \in \{i_1,\dots,i_r\}} \prod_{j \notin \{i_1,\dots,i_r\}} \prod_{m=0}^{d_i}(1+y{\frac{t_i}{t_j}}q^{m}),
    \]
    \[
    B = \prod_{i \in \{i_1,\dots,i_r\}}\prod_{m=1}^{d_i}(1+yq^{m}),
    \]
    \[
    C = \prod_{i,j \in \{i_1,\dots,i_r\}}^{i \ne j}\prod_{m=-d_i}^{-d_i+d_j-1}(1+y{\frac{t_i}{t_j}}q^{-m})
    \]
\end{Corollary}
We will use this expression in the next section when we discuss $J$-function.
\section{J-function}\label{J-function section}
 \subsection{$J$-function}
 We begin by recalling the definitions of the moduli space of stable maps and the $K$-theoretic $J$-function (see \cite{Taipale}). Our goal is to establish a connection between the $J$-function, the Quot scheme, and the computations in the previous section. To achieve this, we employ several key results from \cite{Bertram:2003qd}.

 Let $X$ be a nonsingular projective variety. A map from a projective, nodal, connected, $n$-pointed curve of arithmetic genus zero $(C, p_1, \dots, p_n)$ to $X$ is called stable if the following conditions hold:
 \begin{itemize}
     \item The marked points $p_1,\dots,p_n$ are nonsingular;
     \item If an irreducible component of $C$ maps to a point in $X$, it must have at least three special points, which are either marked points or nodes.
 \end{itemize}
 For a nonsingular projective variety $X$ and $\beta \in H_2(X,\mathbb{Z})$, there exists a moduli space $\overline{M}_{0,n}(X,\beta)$ parametrizing stable maps from $n$-pointed, genus-zero curves of class $\beta$; that is, $f_*([C]) = \beta$.
Of particular interest is the graph space $G_{X,\beta} := \overline{M}_{0,0}(X \times \mathbb{P}^1,(\beta,1))$. 
 The action \eqref{action} naturally induces a $\mathbb{C}_q^*$-action on $G_{X,\beta}$. Its fixed locus consists of several connected components, one of which is $F:=\overline{M}_{0,1}(X,\beta)$; see \cite[Section 1]{Bertram:2003qd}.
 As we said before, in this paper we focus on $X=G(r,n)$. 
 \begin{definition}
     The $K$-theoretic $J$-function is defined as the power series:
 \begin{align*}
 J^{X,K}(q,Q) = \sum_{d=0}^{\infty}Q^dJ_d^{X,K}(q),
\end{align*}
 where
 \[
 J_d^{X,K}(q) = ev_*\!\left(\frac{\mathcal{O}_{\overline{M}_{0,1}(X,d)}}{\lambda_{-1}(N^*_{F/G_{X,d}})}\right),
 \]

  and $ev:\overline{M}_{0,1}(X,d) \to X$ is the evaluation map at the single marked point.
  \begin{remark}
      Here, $J_d$ denotes the coefficient of the $J$-function. In previous sections, when we discussed the $I$-function, we used $I_d$. Since, for $G(r,n)$, the $I$-function coincides with the $J$-function, we hope this notation will not confuse the readers.
\end{remark}

 \end{definition}
The following diagram, introduced in \cite{Bertram:2003qd}, plays a crucial role in computing the cohomological $J$-function. The same diagram was used by Taipale in \cite{Taipale} to compute the $K$-theoretic $J$-function.

\begin{center}
 \begin{tikzcd}
Quot_{\mathbb{P}^1,d}(\mathbb{C}^n,r)   \arrow{r}{\wedge^r} & \mathbb{P}_d^{\binom{n}{r}-1}   \arrow[dr, phantom] & G_{X,d} \arrow[l]{} \\
\sqcup_{\{d_i\}} i_{d_i}(\mathrm{Fl}) \arrow{u}{\kappa}  \arrow{r}{q}        & \mathbb{P}^{{\binom{n}{r}}-1} \arrow{u}{\iota}                                        & \overline{M}_{0,1}(X ,d) \arrow{u}{i} \arrow{l}{p}
\\
Quot_{\mathbb{P}^1,d}(\mathbb{C}^n,r)_{0}^{T \times \mathbb{C}^*} \arrow{u}{}  \arrow{r}{}        & (\mathbb{P}^{{\binom{n}{r}}-1})^{T} \arrow{u}{}                                        & \overline{M}_{0,1}(X ,d)^{T} \arrow{u}{} \arrow{l}{}
\end{tikzcd}
\end{center}  
Here:
\begin{itemize}
\item $\mathbb{P}_d^{\binom{n}{r}-1}$ is Drinfeld's compactification of $Hom_d(\mathbb{P}^1,G(r,n))$. Via the Plücker embedding, a map $\mathbb{P}^1 \to G(r,n)$ can be viewed as a map $\mathbb{P}^1 \to \mathbb{P}^{\binom{n}{r}-1}$, which is determined by $\binom{n}{r}$ homogeneous polynomials of degree $d$ without a common factor. Drinfeld's compactification relaxes this coprimality condition. When all polynomials share a common factor $x^d$, we recover a copy of $\mathbb{P}^{\binom{n}{r}-1}$, which is fixed under the $\mathbb{C}_q^*$-action.
\item Each space in the first row has a $\mathbb{C}_q^*$-action, and by \Cref{Relative fixed-points} and the discussion above, the second row consists of $\mathbb{C}_q^*$-fixed loci of the first row.
\item Each space in the first row also admits a $T$-action, and the third row lists the components fixed under $T \times \mathbb{C}_q^*$.
\end{itemize}
From the previous section, we know that the fixed locus $Quot_{\mathbb{P}^1,d}(\mathbb{C}^n,r)_{0}^{T \times \mathbb{C}_q^*}$ is a finite set of points, as described in \eqref{E:quotfixedpt}.
We recall the following lemma (Lemma 6 in \cite{Taipale}):
\begin{lemma} \label{function-defined}
    $J_d^{G(r,n),K}(q) = \sum_{\{d_i\}}\rho_*\left(\frac{1}{\lambda_{-1}(T^*_{i_{\{d_i\}} \mathrm{Fl}})}\right)$,
    where 
    \begin{align}\label{map:rho}
        \rho:  \sqcup_{\{d_i\}} i_{d_i}(\mathrm{Fl}) \to G(r,n)
    \end{align}
     is a morphism whose restriction to each flag variety is the natural projection from the flag bundle $\mathrm{Fl}(m_1,\dots,m_k,r;n) \to G(r,n)$. 
\end{lemma}

The following result, known as the correspondence of residues, will also be used:
\begin{lemma}\label{residue}
    Let $X$ and $Y$ be nonsingular  $T$-schemes, and let $g:X \to Y$ be a proper, $T$-equivariant morphism. $W = \{W_k\}$, and $V$ are components of the $T$-fixed loci of $X$, and $Y$ respectively. We summarize the situation as follows:
\begin{equation*}
    \begin{tikzcd} 
  W \ar[r, "\iota"] \ar[d, "f"] & X \ar[d, "g"] \\
  V \ar[r, "\iota'"] & Y
\end{tikzcd}
\end{equation*}
Then, for $\alpha \in K_T(X)$, we have:
\[
\frac{\iota'^*g_*(\alpha)}{\lambda_{-1}N^{\vee}_{V/Y}} = f_*\sum_{k}\frac{\iota^*\alpha}{\lambda_{-1}N^{\vee}_{\{W_k\}/X}}.
\]
\end{lemma}

From these two lemmas, we obtain the following:

\begin{Corollary}\label{Khaste}
     The restriction of $J_d^{G(r,n),K}$ to the $T$-fixed point $\langle e_{i_1}, \dots,e_{i_r}\rangle$ of $G(r,n)$ is:
     \[
    {J_d^{G(r,n),K}}\arrowvert_{\langle e_{i_1}, \dots,e_{i_r}\rangle} = \lambda_{-1}T^*_{\langle e_{i_1}, \dots,e_{i_r}\rangle}G(r,n)\sum_{\substack{x \in Quot_{\mathbb{P}^1,d}(\mathbb{C}^n,r)_{0}^{T \times \mathbb{C}_q^*} \\ \rho(x) = \langle e_{i_1}, \ldots, e_{i_r} \rangle}} \frac{1}{\lambda_{-1}T^*_{x}Quot}.
    \]
    Equivalently:
    \[
     {J_d^{G(r,n),K}}\arrowvert_{\langle e_{i_1}, \dots,e_{i_r}\rangle} =\prod_{i \in \{i_1,\dots,i_r\} } \prod_{j \notin \{i_1,\dots,i_r\} }(1-\frac{t_i}{t_j})  \sum_{\substack{x \in Quot_{\mathbb{P}^1,d}(\mathbb{C}^n,r)_{0}^{T \times \mathbb{C}_q^*} \\ \rho(x) = \langle e_{i_1}, \ldots, e_{i_r} \rangle}} \frac{1}{\lambda_{-1}T^*_{x}Quot}
     \]
     
    \end{Corollary}
    \begin{remark}
        Note that by \Cref{Corollary2.5},  $\lambda_{-1}T_{x}^*Quot$ depends on non-negative integers $\{d_i\}$, where $\sum_{i=1}^r d_i = d$. 
    \end{remark}
    \begin{proof}
    To prove the corollary, we apply the correspondence of residues to the morphism $\rho$ from \Cref{function-defined}. We define a trivial $\mathbb{C}_q^*$-action on $G(r,n)$, so $T \times \mathbb{C}_q^*$ acts on both schemes involved in $\rho$. Consider the diagram:
    \begin{equation*}
    \begin{tikzcd} 
    \rho^{-1}(\langle e_{i_1}, \dots, e_{i_r}\rangle) 
        \ar[r, "\iota"] 
        \ar[d, "f"] 
    &  \displaystyle \bigsqcup_{\{d_i\}} i_{d_i}(\mathrm{Fl}) 
        \ar[d, "\rho"] \\
    \{\langle e_{i_1}, \dots, e_{i_r}\rangle\} 
        \ar[r, "\iota'"]
    & G(r,n)
\end{tikzcd}
\end{equation*}
Applying \Cref{residue} yields: 
    
       \begin{equation}\label{4} 
\rho_*\frac{1}{\lambda_{-1}(N^*_{ \sqcup_{\{d_i\}} i_{d_i}(\mathrm{Fl})/Quot})}\arrowvert_{\langle e_{i_1}, \dots,e_{i_r} \rangle} = \lambda_{-1}T^*_{\langle e_{i_1}, \dots,e_{i_r} \rangle}G(r,n) \sum_{x \in  (\sqcup_{\{d_i\}} i_{d_i}(\mathrm{Fl}))^T,\rho(x)=\langle e_{i_1}, \dots,e_{i_r} \rangle}\frac{\frac{1}{\lambda_{-1}(N^*_{ \mathrm{Fl}/Quot})}\arrowvert_{x}}{\lambda_{-1}(T_{x}^* \mathrm{Fl})}.
\end{equation}
By \Cref{Relative fixed-points}, we obtain $(\sqcup_{\{d_i\}} i_{d_i}(\mathrm{Fl}))^T = Quot_{\mathbb{P}^1,d}(\mathbb{C}^n,r)^{T \times \mathbb{C}_q^*}_0$. Let $x$ be a fixed point in the summand, and let $\mathrm{Fl}$ denote the flag variety containing $x$. Let $j: \mathrm{Fl} \hookrightarrow Quot_{\mathbb{P}^1,d}(\mathbb{C}^n,r)$ denote the inclusion. We have the short exact sequence on $\mathrm{Fl}$: 
\[
0 \to T{\mathrm{Fl}} \to j^*T{Quot_{\mathbb{P}^1,d}(\mathbb{C}^n,r)} \to N_{ \mathrm{Fl}/{Quot_{\mathbb{P}^1,d}(\mathbb{C}^n,r)}} \to 0. 
\]
Taking the fiber at $x$, we get:
\[
0 \to T_x \mathrm{Fl} \to T_x{Quot_{\mathbb{P}^1,d}(\mathbb{C}^n,r)} \to N_{{\mathrm{Fl}}/{Quot_{\mathbb{P}^1,d}(\mathbb{C}^n,r)}}\arrowvert_{x} \to 0.
\]
By the Whitney formula, this gives:
\begin{equation} \label{5}
\lambda_{-1}((N_{{\mathrm{Fl}}/{Quot_{\mathbb{P}^1,d}(\mathbb{C}^n,r)}})\arrowvert_x) \cdot \lambda_{-1}(T_x{\mathrm{Fl}}) = \lambda_{-1}(T_x{Quot_{\mathbb{P}^1,d}(\mathbb{C}^n,r)}).
\end{equation}
Now by \Cref{function-defined} and substituting \eqref{5} in \eqref{4}, we obtain:
\begin{align*}
J_d\arrowvert_{\langle e_{i_1}, \dots,e_{i_r} \rangle}
&= \lambda_{-1}T^*_{\langle e_{i_1},\dots,e_{i_r}\rangle}G(r,n)
\sum_{\substack{x \in Quot_{\mathbb{P}^1,d}(\mathbb{C}^n,r)_{0}^{T \times \mathbb{C}_q^*}\\ \rho(x) = \langle e_{i_1},\dots,e_{i_r} \rangle}}
\frac{1}{\lambda_{-1}(T^*_x{Quot_{\mathbb{P}^1,d}(\mathbb{C}^n,r)})}.
\end{align*}

\end{proof}
\begin{remark}\label{Remark}
    The $T \times \mathbb{C}_q^*$-fixed points $x$ of the Quot scheme with  $\rho(x) = \langle e_{i_1},\dots,e_{i_r} \rangle$ correspond exactly to those described in \eqref{E:quotfixedpt}, where $\delta_i \ne 0$ only for $i \in \{i_1,\dots,i_r\}$.

\end{remark}

\subsection{{\bf $\lambda_y$-balanced $K$-theoretic class} of $J$-function} \label{lambda calculation}
Consider the collection of $T \times \mathbb{C}_q^*$-modules:
\[
\Bigl\{T^*_x Quot_{\mathbb{P}^1,d}(\mathbb{C}^r,r) : \text{$d \in \mathbb{Z}$, $x$ is a $T \times \mathbb{C}_q^*$-fixed point supported at zero}\Bigr \}.
\]
his collection will be used to describe the $U$-good property of the $J$-function coefficients. More precisely, the set $U$ is defined as:
\[
U = \Bigl\{\frac{1}{\lambda_{-1}T^*_xQuot_{\mathbb{P}^1,d}(\mathbb{C}^r,r)}: \text{$d \in \mathbb{Z}$, $x$ is a $T \times \mathbb{C}_q^*$-fixed point of supported at zero}\Bigr\}.
\]
In this section, we compute the $\lambda_y$-balanced $K$-theoretic class of the $J$-function restricted to each fixed point of $G(r,n)$, corresponding to the set $U$. We begin with the following lemma:
\begin{lemma}
    Each coefficient of the $J$-function, i.e., $J_d$, is a $U-good$ class(see \Cref{balancing definitions}).
\end{lemma}
\begin{proof}
    By \Cref{Khaste} and \Cref{Corollary2.5}, we obtain:
    \[
{J_d^{G(r,n),K}}\arrowvert_{\langle e_{i_1},\dots,e_{i_r}\rangle} = \sum_{d_{i_1}+\dots+d_{i_r}=d}ABC,
\]
where:
\[
A = \prod_{i \in \{i_1,\dots,i_r\}} \prod_{j \notin \{i_1,\dots,i_r\}} \prod_{m=1}^{d_i} \frac{1}{(1 - \frac{t_i}{t_j} q^{m})},
\]
\[
B = \prod_{i \in \{i_1,\dots,i_r\}}\prod_{m=1}^{d_i}\frac{1}{(1-q^{m})},
\]
\[
C = \prod_{i,j \in \{i_1,\dots,i_r\}}^{i \ne j}\prod_{m=-d_i}^{-d_i+d_j-1}\frac{1}{(1-{\frac{t_i}{t_j}}q^{-m})}.
\]
 Therefore, ${J_d^{G(r,n),K}}\arrowvert_{\langle e_{i_1},\dots,e_{i_r}\rangle}$ is a rational function in $\mathbb{N}\langle U \rangle$.
\end{proof}

Since ${J_d^{G(r,n),K}}\arrowvert_{\langle e_{i_1},\dots,e_{i_r}\rangle}$ is a U-good class, we can find the  $\lambda_y$-balanced $K$-theoretic class of ${J_d^{G(r,n),K}}\arrowvert_{\langle e_{i_1},\dots,e_{i_r}\rangle}$ at each fixed point:
\begin{align} \label{Mohem}
\mathcal{B}_y\!\left({J_d^{G(r,n),K}}\!\arrowvert_{\langle e_{i_1},\dots,e_{i_r}\rangle}\right)
&= \sum_{d_{i_1}+\dots+d_{i_r}=d}
\mathcal{B}_y(A)\,\mathcal{B}_y(B)\,\mathcal{B}_y(C),
\end{align}
where:
\[
\mathcal{B}_y(A) = \prod_{i \in \{i_1,\dots,i_r\}} \prod_{j \notin \{i_1,\dots,i_r\}}\prod_{m=1}^{d_i} \frac{(1+y{\frac{t_i}{t_j}}q^{m})}{(1-{\frac{t_i}{t_j}}q^{m})} ,
\]
\[
\mathcal{B}_y(B) = \prod_{i \in \{i_1,\dots,i_r\}}\prod_{m=1}^{d_i}\frac{(1+yq^{m})}{(1-q^{m})},
\]
\[
\mathcal{B}_y(C) = \prod_{i,j \in \{i_1,\dots,i_r\}}^{i \ne j}\prod_{m=-d_i}^{-d_i+d_j-1}\frac{(1+y{\frac{t_i}{t_j}}q^{-m})}{(1-{\frac{t_i}{t_j}}q^{-m})}.
\]
Moreover, since ${J^{G(r,n),K}}\arrowvert_{\langle e_{i_1},\dots,e_{i_r}\rangle} = \sum_{d \geq 0}Q^d{J_d^{G(r,n),K}}\arrowvert_{\langle e_{i_1},\dots,e_{i_r}\rangle}$, we obtain:
\[
\mathcal{B}_y({J^{G(r,n),K}}\arrowvert_{\langle e_{i_1}, \dots,e_{i_r}\rangle}) = \sum_{d \geq    0}Q^d\mathcal{B}_y({J_d^{G(r,n),K}}\arrowvert_{\langle e_{i_1}, \dots,e_{i_r}\rangle}),
\]
where the coefficient of $Q^d$ is given by \eqref{Mohem}. We will use this power series in \Cref{section 7}.

\section{Cotangent bundle of \texorpdfstring{$G(r,n)$}{Gr(r,n)}}
Starting in this section, we turn to the definition of quasimaps to the cotangent bundle of $G(r,n)$. The cotangent bundle of $G(r,n)$ is a special case of a more general concept known as Nakajima quiver varieties. We briefly recall the definition of the cotangent bundle of $G(r,n)$. For the general definition of Nakajima quiver varieties, we refer the reader to \cite{MR3644099}.

Let $Q$ be a quiver with two vertices and one directed arrow:
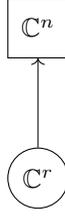
\vspace{0.5em}
\begin{center}
\begin{tikzpicture}[node distance=1.5cm]
    \node[draw, rectangle, minimum width=0.8cm, minimum height=0.8cm] (A) at (0, 0) {$\mathbb{C}^n$};  
    \node[draw, circle, minimum size=0.8cm] (B) at (0, -2) {$\mathbb{C}^r$};  

    \draw[->] (B) -- (A) node[midway, right] {};  
\end{tikzpicture}
\end{center}
\vspace{0.5em}
A representation of the quiver $Q$, denoted $\mathrm{Rep}(Q,\mathbb{C}^r,\mathbb{C}^n)$, is the $GL_r(\mathbb{C})$-representation given by:
\[
\mathrm{Hom}(\mathbb{C}^r,\mathbb{C}^n)
\]
and its cotangent bundle is:
\[
T^*\mathrm{Rep}(Q,\mathbb{C}^r,\mathbb{C}^n) = \mathrm{Hom}(\mathbb{C}^r,\mathbb{C}^n) \oplus \mathrm{Hom}(\mathbb{C}^n,\mathbb{C}^r).
\]
 Define the moment map as follows:
 \[
 \mu: \mathrm{Hom}(\mathbb{C}^r, \mathbb{C}^n) \oplus \mathrm{Hom}(\mathbb{C}^n,\mathbb{C}^r) \to \mathfrak{gl}(\mathbb{C}^r)^* \cong \mathfrak{gl}(\mathbb{C}^r)
 \]
 \[
 \mu(A,B) \mapsto BA
 \]
  where the isomorphism $\mathfrak{gl}(\mathbb{C}^r)^* \cong \mathfrak{gl}(\mathbb{C}^r)$ is given by the trace pairing: for  $X \in \mathfrak{gl}(\mathbb{C}^r)$,
  \[
  \phi_X: Y \mapsto \mathrm{Tr}(XY)
  \]
  with $Y \in \mathfrak{gl}(\mathbb{C}^r)$.
  
  The cotangent bundle of $G(r,n)$ is isomorphic to: 
\[
\mu^{-1}(0)^{ss}/GL_r(\mathbb{C})
\]
where $\mu^{-1}(0)^{ss}$ denotes the subset of $\mu^{-1}(0) \subset T^*\mathrm{Rep}(Q,\mathbb{C}^r,\mathbb{C}^n)$ consisting of pairs $(A,B)$ such that $A$ has full rank $r$:
\[
\mu^{-1}(0)^{ss} = \left\{
(A,B) \in \mathrm{Hom}(\mathbb{C}^r,\mathbb{C}^n) \oplus \mathrm{Hom}(\mathbb{C}^n,\mathbb{C}^r)
 \biggm| \begin{array}{l}
A \textrm{ is injective}\\
 \textrm{$BA=0$}
\end{array}
\right\}.
\]
See \cite[Example 10]{MR3644099}.
\section{Quasimaps to \texorpdfstring{$T^*G(r,n)$}{T*Gr(r,n)}}\label{section 7}
\subsection{Definition of stable quasimaps}
Recall that in \cref{sec:Quot scheme}, we introduced the Quot scheme as a compactification of the moduli space of maps to the Grassmannian. In this section, we focus on the moduli space of maps from $\mathbb{P}^1$ to the cotangent bundle of the Grassmannian. We then introduce a new moduli space, known as the space of stable quasimaps. Although this discussion can be extended to maps to GIT quotients (see \cite{CKM}), we focus on the cotangent bundle of the Grassmannian.

 Let $X=T^*G(r,n)$, and let $u: \mathbb{P}^1 \to X$ be a map. As discussed in the previous section, $T^*G(r,n)$ can be viewed as $\mu^{-1}(0)^{ss}/GL_r(\mathbb{C})$, which is an open subset of the quotient stack $[\mu^{-1}(0)/GL_r(\mathbb{C})]$. By considering the composition, we obtain a map to the quotient stack:
\[
u: \mathbb{P}^1 \to  [\mu^{-1}(0)/GL_r(\mathbb{C})].
\]
 By the definition of the quotient stack, any such map is equivalent to
 a vector bundle $\mathcal{P}$ on $\mathbb{P}^1$, together with a section $f$ of the associated vector bundle $\mathcal{P} \times_{GL_r(\mathbb{C})} \mu^{-1}(0)$. Since $\mu^{-1}(0) \subset \mathrm{Hom}(\mathbb{C}^r,\mathbb{C}^n) \oplus \mathrm{Hom}(\mathbb{C}^n,\mathbb{C}^r)$ and $GL_r(\mathbb{C})$ acts only on $\mathbb{C}^r$, the bundle $\mathcal{P}\times_{GL_r(\mathbb{C})}\mu^{-1}(0)$ gives rise to a rank-$r$ vector bundle $\mathcal{V}$ and a trivial rank-$n$ vector bundle $\mathcal{W}$ on $\mathbb{P}^1$. The section $f$ is a section of a bundle $\mathcal{H}om(\mathcal{V},\mathcal{W}) \oplus \mathcal{H}om(\mathcal{V},\mathcal{W})^{\vee}$ and satisfies $\mu = 0$. 
 \begin{definition}
     A quasimap $\mathbb{P}^1 \to \mathcal{X}:= [\mu^{-1}(0)/GL_r(\mathbb{C})]$ consists of the triple:
 \[
 \langle \mathcal{V},\mathcal{W},f \rangle
 \]
 where:
 \begin{itemize}
    \item $\mathcal{V}$ is a vector bundle  of rank $r$ over $\mathbb{P}^1$

    \item $\mathcal{W}$ is a trivial vector bundle of rank $n$ over $\mathbb{P}^1$. 

    \item A section $f \in H^0(\mathbb{P}^1, \mathcal{M} \oplus \hbar \mathcal{M}^{\vee})$, satisfying $\mu =0 $, where $\mathcal{M} =  \mathcal{H}om(\mathcal{V},\mathcal{W})$.
\end{itemize}
 \end{definition}
Here, $\hbar$ denotes twisting by the line bundle of weight $\hbar$. The degree of a quasimap $f$ is defined as the degree $\deg(\mathcal{V}) \in \mathbb{Z}$. We denote a quasimap by its section $f$ and the stack of all degree $d$ quasimaps from $\mathbb{P}^1$ to $X$ by $QMap^d(X)$. Define $QMap(X) = \bigsqcup_{d} QMap^d(X)$.

For each $p \in \mathbb{P}^1$, there exists an evaluation map:
\[
ev_p: QMap(X) \to \mathcal{X}=[\mu^{-1}(0)/GL_r(\mathbb{C})]
\]
given by:
\[
ev_p(f) := f(p).
\]
\begin{definition}
    Let $f$ be a quasimap. A point $p \in \mathbb{P}^1$ is called nonsingular if $f(p) \in X$. A quasimap $f$ is called stable if it is nonsingular at all but finitely many points of $\mathbb{P}^1$.
\end{definition}

\begin{definition}
Let $QMap_{p_1,\dots,p_r}^d(X)$ denote the stack parameterizing 
the data of degree $d$ quasimaps to $T^*G(r,n)$, nonsingular at $p_1,\dots,p_r$.
\end{definition}
As before, let $T = (\mathbb{C}^*)^n$ act as in \eqref{Torus action}, and let $\mathbb{C}_q^*$ act on $\mathbb{P}^1$ as in \eqref{action}. 
Let $\mathbb{C}^*_{\hbar}$ be a torus that acts on $T^*G(r,n)$ by scaling the cotangent direction with the weight $\hbar$. Define $A=T \times \mathbb{C}^*_{\hbar} \times \mathbb{C}_q^*$.

The $T \times \mathbb{C}^*_{\hbar}$-fixed points of $T^*G(r,n)$ correspond to the $\binom{n}{r}$ points, each corresponding to an $r$-dimensional coordinate subspace of $\mathbb{C}^n$. Let these fixed points be denoted by $p_i$ for $1 \leq i \leq {\binom{n}{r}}$.

\subsection{Fixed points of stable quasimap}
In this section, we describe the $A$-fixed points of $QM_{\infty}^d(T^*G(r,n))$, the space used to define the vertex function in Section \eqref{vertex function}. The following lemma can be found in \cite[Section 4.4]{PSZ}, but for the reader's convenience, we provide the details here. For simplicity of notation, we assume that $f(\infty) = p_1 = \langle e_1, \dots, e_r \rangle$; the proof will explain why this assumption is without loss of generality.
\begin{lemma}\label{lemma:qmap-fixed}
    The $A$-fixed points of $QM_{\infty}^d(T^*G(r,n))$ are given by the data $\langle\mathcal{V},\mathcal{W},f\rangle$, where:
    \begin{align*}
    &\mathcal{W} = \mathbb{C}_{t_1}\otimes\mathcal{O}_{\mathbb{P}^1} \oplus\dots\oplus \mathbb{C}_{t_n}\otimes \mathcal{O}_{\mathbb{P}^1}\\&
    \mathcal{V} = q^{d_1}\mathbb{C}_{t_1}\otimes\mathcal{O}_{\mathbb{P}^1}(-d_1) \oplus \dots \oplus q^{d_r}\mathbb{C}_{t_r}\otimes\mathcal{O}_{\mathbb{P}^1}(-d_r) \\&
    f = (x^{d_1},\dots,x^{d_r})\\&
    \sum_{i=1}^{r} d_i = d
\end{align*}
\end{lemma}

\begin{proof}
    
 Assume that $(\mathcal{V},\mathcal{W},f)$ 
is an $A$-fixed point of the degree $d$ stable quasimap space $QM_{\infty}^d(T^*G(r,n))$. Since $\mathcal{W}$ is trivial and $T=(\mathbb{C}^*)^n$ acts on it, we may write it equivariantly as:
\begin{equation}\label{first}
\mathcal{W}=\mathbb{C}_{t_1}\otimes\mathcal{O}_{\mathbb{P}^1} \oplus \dots \oplus \mathbb{C}_{t_n}\otimes\mathcal{O}_{\mathbb{P}^1}.
\end{equation}
Since $f$ is a fixed point, $f(\infty)$ must be a $T \times \mathbb{C}^*_{\hbar}$-fixed point of $T^*G(r,n)$ and by the assumption, $f(\infty) = p_1 = \langle e_1,\dots,e_r \rangle$. Since we have:
\[
f(\infty) = p_1 \in T^*G(r,n)=\left\{
(A,B) \in Hom(\mathbb{C}^r,\mathbb{C}^n) \oplus Hom(\mathbb{C}^n,\mathbb{C}^r)
 \biggm| \begin{array}{l}
A \textrm{ is injective,}\\
 \textrm{ $BA=0$}
\end{array}
\right\}/GL_r(\mathbb{C}),
\]
the component of $f$ in $Hom(\mathcal{V},\mathcal{W})$ must be an injection. Moreover, $\mathcal{V}$ must have weights $t_1,\dots,t_r$ because $p_1$ corresponds to the subspace spanned by $e_1,\dots,e_r$. Since the quasimap is fixed by $T$, we can apply the argument from \Cref{Torussection} for the Quot scheme. Hence, $\mathcal{V}$ is of the form:

\[
 \mathbb{C}_{t_1}\otimes\mathcal{O}_{\mathbb{P}^1}(-d_1) \oplus \dots \oplus \mathbb{C}_{t_r}\otimes\mathcal{O}_{\mathbb{P}^1}(-d_r) 
\]
and the component of $f$ in $Hom(\mathcal{V},\mathcal{W})$ is given by inclusions:
\begin{equation}\label{given maps}
\mathbb{C}_{t_i}\otimes\mathcal{O}_{\mathbb{P}^1}(-d_i) \to \mathbb{C}_{t_i}\otimes\mathcal{O}_{\mathbb{P}^1} \/, \quad 1 \leq i \leq r,
\end{equation}
which is induced by multiplication by a homogeneous polynomial $f_i$ of degree $d_i$, with $d_i \ge 0$ and $d_1+\dots+d_r=d$.

Moreover, since $d_i \geq 0$, the component $\hbar^{-1}Hom(\mathcal{W},\mathcal{V})$ is zero. This is because it is (non-equivariantly) a direct sum of terms of the form $Hom(\mathcal{O}_{\mathbb{P}^1} , \mathcal{O}_{\mathbb{P}^1}(-d_i)) \cong H^0(\mathbb{P}^1,\mathcal{O}_{\mathbb{P}^1}(-d_i))$, which is zero when $d_i \geq 0$.

Let $f=(f_1,\dots,f_r)$. Since $f$ must be $\mathbb{C}_q^*$-fixed, and again by the same argument as in Section \eqref{Torussection} for the Quot scheme, each $f_{i}$ must be a monomial in $x$ and $y$. 

Moreover, since $f(\infty) = p_1$, each component $f_i$ of $f$ must be nonzero at $\infty = [1:0]$, and the only homogeneous polynomial of degree $d_i$ nonzero at $[1:0]$ is $x^{d_i}$. So, we obtain:
\begin{equation}\label{second}
    f=(x^{d_1},\dots,x^{d_r}).
\end{equation}
Now, by the action of $\mathbb{C}_
q^*$ we have: $q\cdot x^{d_i} = q^{d_i}x^{d_i}$. Therefore, for this section to be fixed, we need to twist each $\mathbb{C}_{t_i}\otimes\mathcal{O}_{\mathbb{P}^1}(-d_i)$ by $q^{d_i}$. Therefore:
\begin{equation}\label{third}
    \mathcal{V} = q^{d_1}\mathbb{C}_{t_1}\otimes\mathcal{O}_{\mathbb{P}^1}(-d_1) \oplus \dots \oplus q^{d_r}\mathbb{C}_{t_r}\otimes\mathcal{O}_{\mathbb{P}^1}(-d_r).
\end{equation}

 \eqref{first}, \eqref{second}, and \eqref{third} complete the proof.
\end{proof}
\begin{remark}
    In the lemma, since $f=(x^{d_1},\dots,x^{d_r})$, the maps in \eqref{given maps} correspond explicitly to multiplication by $x^{d_i}$.
\end{remark}

\subsection{Bijection of fixed points of stable quasimaps and Quot scheme}
In this short section, we prove that there is a bijection between the $A=T \times \mathbb{C}^*_q \times \mathbb{C}^*_\hbar$-fixed points of degree $d$ stable quasimaps and $T \times \mathbb{C}_q^*$-fixed points of the Quot scheme, $Quot_{\mathbb{P}^1,d}(\mathbb{C}^n,r)$, supported at zero. Recall from \Cref{sec:Quot scheme} that this set is denoted $Quot_{\mathbb{P}^1,d}(\mathbb{C}^n,r)_{0}^{T \times \mathbb{C}_q^*}$. The following corollary describes this bijection explicitly:
\begin{Corollary}
    There exists an explicit bijection: 
\[\phi: QM_{\infty}^d(T^*G(r,n))^{A} \to Quot_{\mathbb{P}^1,d}(\mathbb{C}^n,r)_0^{T \times \mathbb{C}_q^*}
\]
given by:
\[
\phi(\langle \mathcal{V},\mathcal{W},f\rangle) = \mathbb{C}_{t_1}\otimes\mathcal{O}_{\mathbb{P}^1}(-d_1) \oplus \dots \oplus \mathbb{C}_{t_r}\otimes\mathcal{O}_{\mathbb{P}^1}(-d_r) \to \mathbb{C}_{t_1}\otimes\mathcal{O}_{\mathbb{P}^1} \oplus \dots \oplus \mathbb{C}_{t_n}\otimes\mathcal{O}_{\mathbb{P}^1},
\]
where $\langle \mathcal{V},\mathcal{W},f\rangle$ is a fixed point as in \Cref{lemma:qmap-fixed}, and the right hand side is given by \eqref{E:quotfixedpt}.

\end{Corollary}
\begin{proof}
From \Cref{lemma:qmap-fixed}, the $A$-fixed points of $QM_{\infty}^d(T^*G(r,n))$ correspond to triples $\langle\mathcal{W},\mathcal{V},f\rangle$, where:
    \begin{align*}
    &\mathcal{W} = \mathbb{C}_{t_1}\otimes\mathcal{O}_{\mathbb{P}^1} \oplus\dots\oplus \mathbb{C}_{t_n}\otimes\mathcal{O}_{\mathbb{P}^1},\\&
    \mathcal{V} = q^{d_1}\mathbb{C}_{t_1}\otimes\mathcal{O}_{\mathbb{P}^1}(-d_1) \oplus \dots \oplus q^{d_r}\mathbb{C}_{t_r}\otimes\mathcal{O}_{\mathbb{P}^1}(-d_r),\\&
    f = (x^{d_1},\dots,x^{d_r}).
\end{align*}

Under $\phi$, this fixed point maps to the point of $Quot_{\mathbb{P}^1,d}(\mathbb{C}^n,r)_{0}^{T \times \mathbb{C}_q^*}$ given by: 
    
    \begin{align*}
        \mathbb{C}_{t_1}\otimes\mathcal{O}_{\mathbb{P}^1}(-d_1) \oplus \dots \oplus \mathbb{C}_{t_r}\otimes\mathcal{O}_{\mathbb{P}^1}(-d_r) \to \mathbb{C}_{t_1}\otimes\mathcal{O}_{\mathbb{P}^1} \oplus \dots \oplus \mathbb{C}_{t_n}\otimes\mathcal{O}_{\mathbb{P}^1}
    \end{align*}
where the map $\mathbb{C}_{t_i}\otimes\mathcal{O}_{\mathbb{P}^1}(-d_i) \to \mathbb{C}_{t_i}\otimes\mathcal{O}_{\mathbb{P}^1}$ is given by multiplication by $x^{d_i}$ as described in \eqref{E:quotfixedpt}.
Since the construction is reversible, this map is indeed a bijection, as desired.
 
\end{proof}

\subsection{Vertex Function}\label{vertex function}
The vertex function, introduced by Okounkov \cite{MR3752463}, provides an equivariant count of quasimaps from $\mathbb{P}^1$ to $T^*G(r,n)$. In fact, this construction applies more generally to any Nakajima quiver variety.

It is known that the moduli space of degree $d$ stable quasimaps, nonsingular at $p \in \mathbb{P}^1$, denoted by $QM_p^d( X)$, carries a perfect obstruction theory \cite[Theorem 7.2.2]{CKM}. This implies the existence of a virtual tangent bundle $T_{\mathrm{vir}}$ and a virtual structure sheaf $\mathcal{O}^{\mathrm{vir}}_d$. 
In this theory, the fiber of the virtual tangent bundle at a fixed point $\langle\mathcal{V}, \mathcal{W}, f\rangle$ is given by:
\begin{equation}\label{Virtual tangent bundle}
{T_{\mathrm{vir},\langle\mathcal{V}, \mathcal{W},f\rangle}}
 = H^{\bullet}(T_{\langle \mathcal{V},\mathcal{W},f \rangle}^{1/2} \oplus \hbar(T_{\langle \mathcal{V},\mathcal{W},f \rangle}^{1/2})^{\vee})- T_{p}X,
\end{equation}
where
\[
T_{\langle \mathcal{V},\mathcal{W},f \rangle}^{1/2} = \mathcal{V}^{\vee} \otimes \mathcal{W} - \mathcal{V}^{\vee} \otimes \mathcal{V},
\]
and $H^{\bullet} = H^0 - H^1$ in $K$-theory.

 Okounkov also defined the {\bf symmetrized virtual structure} sheaf by:

\begin{align*}
  \hat{\mathcal{O}}^{{\mathrm{vir}}}_d =  \mathcal{O}^{\mathrm{vir}}_d \otimes \mathcal{K}_{\mathrm{vir}}^{\frac{1}{2}},
\end{align*}
 where $\mathcal{K}_{\mathrm{vir}}^{\frac{1}{2}}$ denotes the square root of the virtual canonical line bundle. For further details, see \cite{MR3752463}.
 
     At any point $\langle \mathcal{V}, \mathcal{W}, f \rangle,$ the virtual tangent bundle can be written as:
     \[
     T_{\mathrm{vir},\langle \mathcal{V},\mathcal{W},f \rangle} = \sum a_i - \sum b_i,
     \]
    where the $a_i$ correspond to the contributions from $H^0$, and the $b_i$ correspond to those from $H^1$ and $T_pX$ in ~\eqref{Virtual tangent bundle}. 
     
     The virtual canonical bundle at any point $\langle \mathcal{V},\mathcal{W},f \rangle$ is therefore given by: 
     \[
     K_{\mathrm{vir},\langle \mathcal{V},\mathcal{W},f \rangle}=(\mathrm{det}T_{\mathrm{vir},\langle \mathcal{V},\mathcal{W},f \rangle})^{-1} = \frac{\prod b_i}{\prod a_i}.
     \]
     Moreover,  note that 
     \[
     \lambda_{-1}T_{\mathrm{vir},\langle \mathcal{V},\mathcal{W},f \rangle}^* = \prod\frac{1 - a_i^{-1}}{1 - b_i^{-1}},
     \]
     so that: 
     \[
     \frac{{K_{\mathrm{vir},\langle \mathcal{V},\mathcal{W},f \rangle}}^{\frac{1}{2}}}{\lambda_{-1}T_{\mathrm{vir},\langle \mathcal{V},\mathcal{W},f \rangle}^*} = \prod \frac{b_i^{\frac{1}{2}}-b_i^{-\frac{1}{2}}}{a_i^{\frac{1}{2}}-a_i^{-\frac{1}{2}}}.
     \]
     This term appears in the localization calculation of classes involving $\mathcal{K}_{\mathrm{vir}}$. Okounkov's roof function is defined by:
     \[
     \hat{s}(x) = \frac{1}{x^{\frac{1}{2}}-x^{-\frac{1}{2}}},
     \]
     and satisfies the following properties: 
     \[
     \hat{s}(x+y) = \hat{s}(x)\hat{s}(y), \quad   \hat{s}(x-y)= \frac{\hat{s}(x)}{\hat{s}(y)}.
     \]
    This function naturally appears in localization formulas involving symmetrized virtual structure sheaves.
    
In \cite[Section 7.2]{MR3752463}, it is proved that there exists a pushforward map $ev_{{\infty},*}:K_{T \times \mathbb{C}^*_q \times \mathbb{C}^*_{\hbar}}(QM^d_{\infty}) \to K_{T \times \mathbb{C}^*_q \times \mathbb{C}^*_{\hbar}}(T^*G(r,n))_{loc}$, which allows us to define the quasimap vertex function:

\begin{definition} \label{vertex def}
    The {\bf quasimap vertex function} is defined by:
    \[
 \sum_{d \geq 0}ev_{{\infty}, *}(QM^d_{\infty}, \hat{\mathcal{O}}^{{\mathrm{vir}}}_d Q^d) \in K_{T \times \mathbb{C}^*_q \times \mathbb{C}^*_{\hbar}}(T^*G(r,n))_{loc}[[Q]].
    \]
\end{definition}

By the virtual localization theorem \cite{FK}, the pushforward can be computed as:
\begin{equation}\label{lastsection}
  ev_{{\infty},*}(QM^d_{\infty},\hat{\mathcal{O}}^{{\mathrm{vir}}}_d Q^d)= \sum_{i=1}^{\binom{n}{r}}\sum_{{\langle \mathcal{V}, \mathcal{W},f \rangle_i^d}}\frac{\hat{\mathcal{O}}^{\mathrm{vir}}_d \vert_{{\langle \mathcal{V}, \mathcal{W},f \rangle_i^d}}}{{{\lambda_{-1}T_{{\mathrm{vir}}}^*}}_{{\langle\mathcal{V}, \mathcal{W},f \rangle_i^d}}} [p_i] Q^d = \sum_{i=1}^{\binom{n}{r}}\sum_{{\langle \mathcal{V}, \mathcal{W},f \rangle_i^d}}\hat{s}(T^{\mathrm{vir}}_{{\langle\mathcal{V}, \mathcal{W},f \rangle_i^d}}) [p_i] Q^d,
\end{equation}
where $\langle \mathcal{V}, \mathcal{W},f \rangle_i^d$ denotes the degree $d$ quasimaps sending $\infty$ to the $i$-th fixed point $p_i$ of $T^*G(r,n)$, equivalently, those quasimaps for which $f(\infty) = p_i$.
 
The quasimap vertex function, as calculated in \cite[Proposition 4]{PSZ} using the localization theorem, is given as follows, where for simplicity we present it only for the fixed point corresponding to the first $r$ coordinate vectors $\langle e_{1},\dots,e_{r}\rangle$ :
    \begin{thm} [Pushkar-Smirnov-Zeitlin]
        \label{lemma:proj} Let $p$ be a torus fixed point in $T^*G(r,n)^{T \times \mathbb{C}_{\hbar}^*}$ corresponding to the first $r$ coordinate vectors $\langle e_1,\dots,e_r \rangle$. The coefficient of the vertex function at this fixed point is the following $q$-hypergeometric function:
\begin{align*}
    V_p(t;q,Q,\hbar) = \sum_{d_1,\dots,d_r \in \mathbb{Z}^{\geq 0}} Q^d q^{nd/2} \prod_{i,j=1}^r \{\frac{t_j}{t_i}\}_{d_i-d_j}^{-1}\prod_{i=1}^{r}\prod_{j=1}^{n}\{\frac{t_j}{t_i}\}_{d_i},
\end{align*}
where $d = \sum_{i=1}^r{d_i}$, $\{x\}_d = \frac{(\frac{\hbar}{x})_d}{(\frac{q}{x})_d}(-q^{\frac{1}{2}}\hbar^{\frac{-1}{2}})^d$, where $(x)_d = (1-x)\cdots(1-xq^{d-1})$. 
\end{thm}
After absorbing the factor $(\frac{-q}{\hbar^{1/2}})^n$ into the quantum parameter $Q$, i.e., under the change of variables $Q \mapsto (\frac{-q}{\hbar ^{1/2}})^{n}Q$, the quasimap vertex function can be rewritten as:
\[
V_p(t;q,Q,\hbar) = \sum_{d_1,\dots,d_r \in \mathbb{Z}^{\geq 0}} Q^d \prod_{i,j=1}^r \{\frac{t_j}{t_i}\}_{d_i-d_j}^{-1}\prod_{i=1}^{r}\prod_{j=1}^{n}\{\frac{t_j}{t_i}\}_{d_i},
\]
where, under this change, $\{x\}_d$ simplifies to $\frac{(\frac{\hbar}{x})_d}{(\frac{q}{x})_d}$.

\subsection{Main Result}\label{Main results}
In this section, we state and prove the main result of the paper. We demonstrate how the quasimap vertex function can be recovered from the $\lambda_y$-balanced $K$-theoretic class of the $I$-function of $G(r,n)$ introduced in \Cref{Balanced class}. We also introduce the $I$-function of $T^*G(r,n)$, defined in \cite{wen2019ktheoreticifunctionvthetamathbfg}, which serves as the analogue of the $J$-function in Gromov--Witten theory. Next, we define its twisted version and prove that the vertex function coincides with the twisted $I$-function of $T^*G(r,n)$. 

 We begin by proving the first result. For simplicity, in the following theorem, we assume that $p$ is the fixed point of both $G(r,n)$ and $T^*G(r,n)$ corresponding to the span of the first $r$ coordinate vectors, $\langle e_{1},\dots,e_{r}\rangle$.
 
\begin{thm} \label{main thm}
     After the change of variables $y=-q^{-1}\hbar$, we have: 
    \[
    V_p(t;q,Q,\hbar) = \mathcal{B}_{-q^{-1}\hbar}(I(t;q,Q)_{\arrowvert_{p}})
    \]

    where the coefficients of $Q^d$ in both functions are equal for every $d \geq 0$, namely,
    \[
    V_p^d(t;q,\hbar) = \mathcal{B}_{-q^{-1}\hbar}(I_d(t;q)_{\arrowvert_{p}}).
    \]
\end{thm}
\begin{proof}
     In the expression for $V_p^d$, we have terms of the form $\prod_{i=1}^{r}\prod_{j=1}^{n}\{\frac{t_j}{t_i}\}_{d_i}$. We divide this term into two parts:

   \begin{itemize}
       \item The case $i=j$:
       
       In this case, we have $r$ terms of $\{1\}_{d_i}$ for $1 \leq i \leq r$. Since $\{1\}_{d_i} = \frac{(1-\hbar)\dots(1-\hbar q^{d_i-1})}{(1-q)\dots(1-q^{d_i})}$, we obtain:
       \[
       \frac{\prod_{i=1}^r \prod_{s=0}^{d_i-1}(1-\hbar q^s)}{\prod_{i=1}^r \prod_{s=1}^{d_i}(1-q^s)}.
       \]
       Now consider $\mathcal{B}_{y}(I(t;q,Q)_{\arrowvert_{p}})$, which contains the terms: 
       $\frac{\prod_{i=1}^r\prod_{s=1}^{d_i}(1+yq^{s})}{\prod_{i=1 }^r\prod_{s=1}^{d_i}(1-q^{s})}$.
       
       The corresponding term in $\mathcal{B}_{q^{-1}y}(I(t;q,Q)_{\arrowvert_{p}})$ will be:
       \[
       \frac{\prod_{i=1}^r\prod_{s=1}^{d_i}(1+yq^{s-1})}{\prod_{i=1 }^r\prod_{s=1}^{d_i}(1-q^{s})} = \frac{\prod_{i=1}^r\prod_{s=0}^{d_i-1}(1+yq^{s})}{\prod_{i=1 }^r\prod_{s=1}^{d_i}(1-q^{s})}.
       \]
       
       Hence, after the change of variable $y=-\hbar$, the terms are equal in $V_p^d$ and $\mathcal{B}_{-q^{-1}\hbar}(I(t;q,Q)_{\arrowvert_{p}})$.

       \item The case $1 \leq i \leq r$ and $r+1 \leq j \leq n$.
       
       These terms in $V_p^d$ are of the form: $\prod_{i=1}^{r}\prod_{j=r+1}^{n}\{\frac{t_j}{t_i}\}_{d_i}$, and we have:
       \[
       \{\frac{t_j}{t_i}\}_{d_i} = \frac{(1-\hbar \frac{t_i}{t_j})\dots(1- \hbar q^{d_i-1}\frac{t_i}{t_j})}{(1-q \frac{t_i}{t_j})\dots(1-q^{d_i}\frac{t_i}{t_j})}.
       \]
       Thus, the overall expression in $V_p^d$ becomes:
       \[
\prod_{i=1}^r \prod_{j=r+1}^{n} \frac{(1-\hbar \frac{t_i}{t_j})\dots(1- \hbar q^{d_i-1}\frac{t_i}{t_j})}{(1-q \frac{t_i}{t_j})\dots(1-q^{d_i}\frac{t_i}{t_j})}   
       \]

       Now, we examine $\mathcal{B}_{-q^{-1}\hbar}(I(t;q,Q)_{\arrowvert_{p}})$. In $\mathcal{B}_{y}(I(t;q,Q)_{\arrowvert_{p}})$, we have the terms: 
       \[
       \frac{\prod_{i=1}^r \prod_{j=r+1}^n\prod_{s=1}^{d_i}(1+yq^{s}\frac{t_i}{t_j})}{\prod_{i=1}^r \prod_{j=r+1}^n\prod_{s=1}^{d_i}(1-q^{s}\frac{t_i}{t_j})} 
       \]
       Hence, the corresponding term in $\mathcal{B}_{q^{-1}y}(I(t;q,Q)_{\arrowvert_{p}})$ will be:

       \[
       \frac{\prod_{i=1}^r \prod_{j=r+1}^n\prod_{s=1}^{d_i}(1+yq^{s-1}\frac{t_i}{t_j})}{\prod_{i=1}^r \prod_{j=r+1}^n\prod_{s=1}^{d_i}(1-q^{s}\frac{t_i}{t_j})} = \frac{\prod_{i=1}^r \prod_{j=r+1}^n\prod_{s=0}^{d_i-1}(1+yq^{s}\frac{t_i}{t_j})}{\prod_{i=1}^r \prod_{j=r+1}^n\prod_{s=1}^{d_i}(1-q^{s}\frac{t_i}{t_j})}
       \]
       Thus, after the change of variables $y=-\hbar$, the terms in $V_p^d$ and $\mathcal{B}_{-q^{-1}\hbar}(I(t;q,Q)_{\arrowvert_{p}})$ are also equal.

       \item Finally, in $V_p^d$, we are left with the terms: $\prod_{i,j=1}^r\{\frac{t_j}{t_i}\}_{d_i}$ where $1 \leq i \ne j \leq r$, and the terms $\prod_{i,j=1}^r\{\frac{t_j}{t_i}\}_{d_i-d_j}^{-1}$. 
       
       In $\prod_{i,j=1}^r\{\frac{t_j}{t_i}\}_{d_i-d_j}^{-1}$, when $i=j$, we have terms of the form $\{1\}_0$, which are equal to 1. So we can focus on the case when $i \ne j$. We then obtain the following:
       \[
\prod_{i,j=1}^r\{\frac{t_j}{t_i}\}_{d_i}\prod_{1 \le i,j \le r, i \ne j}\{\frac{t_j}{t_i}\}_{d_i-d_j}^{-1} = \prod_{1 \le i,j \le r, i \ne j}\{\frac{t_j}{t_i}\}_{d_i-d_j}^{-1}\{\frac{t_j}{t_i}\}_{d_i}.
       \]
       Now we calculate $\{\frac{t_j}{t_i}\}_{d_i-d_j}^{-1}.\{\frac{t_j}{t_i}\}_{d_i}$:
       
        1. If $d_i-d_j > 0:$

        \begin{align*}
           \{\frac{t_j}{t_i}\}_{d_i-d_j}^{-1}\{\frac{t_j}{t_i}\}_{d_i} &= \frac{\prod_{s=1}^{d_i-d_j}(1-q^s\frac{t_i}{t_j})}{\prod_{s=0}^{d_i-d_j-1}(1-\hbar q^s \frac{t_i}{t_j})} \cdot\frac{\prod_{s=0}^{d_i-1}(1-\hbar q^s\frac{t_i}{t_j})}{\prod_{s=1}^{d_i}(1-q^s \frac{t_i}{t_j})}  \\&= \frac{\prod_{s=d_i-d_j}^{d_i-1}(1-\hbar q^s \frac{t_i}{t_j})}{\prod_{s=d_i-d_j+1}^{d_i}(1-q^s \frac{t_i}{t_j})}
        \end{align*}
        2. If $d_i-d_j < 0:$

        \begin{align*}
            \{\frac{t_j}{t_i}\}_{d_i-d_j}^{-1}\{\frac{t_j}{t_i}\}_{d_i} &= \frac{\prod_{s=d_i-d_j}^{-1}(1-\hbar q^s  \frac{t_i}{t_j})}{\prod_{s=d_i-d_j+1}^0(1-q^s\frac{t_i}{t_j})} \cdot \frac{\prod_{s=0}^{d_i-1}(1- \hbar q^s \frac{t_i}{t_j})}{\prod_{s=1}^{d_i}(1-q^s\frac{t_i}{t_j})}  \\&= \frac{\prod_{s=d_i-d_j}^{d_i-1}(1-\hbar q^s  \frac{t_i}{t_j})}{\prod_{s=d_i-d_j+1}^{d_i}(1-q^s\frac{t_i}{t_j})}.
        \end{align*} 
       In both cases, the expressions are the same, so the total term is:
        \begin{align*}
            \prod_{1 \leq i,j \leq r}^{i \ne j}\frac{\prod_{s=d_i-d_j}^{d_i-1}(1-\hbar q^s \frac{t_i}{t_j})}{\prod_{s=d_i-d_j+1}^{d_i}(1-q^s\frac{t_i}{t_j})}
        \end{align*}
        
        We now turn our attention to $\mathcal{B}_{y}(I(t;q,Q)_{\arrowvert_{p}})$, where we are left with the following terms:
        \begin{align*}
        \frac{\prod_{1 \leq i,j \leq r}^{i \ne j}\prod_{s=-d_i}^{-d_i+d_j-1}(1+yq^{-s}\frac{t_i}{t_j})}{\prod_{1 \leq i,j \leq r}^{i \ne j}\prod_{s=-d_i}^{-d_i+d_j-1}(1-q^{-s}\frac{t_i}{t_j})}.
        \end{align*}
        This can be rewritten as:
        \begin{align*}
        \frac{\prod_{1 \leq i,j \leq r}^{i \ne j}\prod_{s=d_i-d_j+1}^{d_i}(1+yq^{s}\frac{t_i}{t_j})}{\prod_{1 \leq i,j \leq r}^{i \ne j}\prod_{s=d_i-d_j+1}^{d_i}(1-q^{s}\frac{t_i}{t_j})}.
        \end{align*}
        Therefore, the terms in $\mathcal{B}_{q^{-1}y}(I(t;q,Q)_{\arrowvert_{p}})$ are given by:
        \begin{align*}
        \frac{\prod_{1 \leq i,j \leq r}^{i \ne j}\prod_{s=d_i-d_j+1}^{d_i}(1+yq^{s-1}\frac{t_i}{t_j})}{\prod_{1 \leq i,j \leq r}^{i \ne j}\prod_{s=d_i-d_j+1}^{d_i}(1-q^{s-1}\frac{t_i}{t_j})},
        \end{align*}
            which are equal to terms:
            \begin{align*}
                \frac{\prod_{1 \leq i,j \leq r}^{i \ne j}\prod_{s=d_i-d_j}^{d_i-1}(1+yq^{s}\frac{t_i}{t_j})}{\prod_{1 \leq i,j \leq r}^{i \ne j}\prod_{s=d_i-d_j+1}^{d_i}(1-q^{s}\frac{t_i}{t_j})}.
            \end{align*}
            Thus, after the change of variables $y=-\hbar$, the terms in $V_p^d$ and $\mathcal{B}_{-q^{-1}\hbar}(I(t;q,Q)_{\arrowvert_{p}})$ are also equal.
    \end{itemize}
    \end{proof}

We now recall and define the $I$-function of $T^*G(r,n)$ from \cite{wen2019ktheoreticifunctionvthetamathbfg}, which serves as an analogue of the $J$-function in Gromov--Witten theory. We then prove that the quasimap vertex function of $T^*G(r,n)$ coincides with its $I$-function. 
\begin{definition}\label{I-function definition}
    The $I$-function of $T^*G(r,n)$ is:
    \[
I=\sum_{d \geq 0}Q^dev_*(\frac{\mathcal{O}_{F_0}^{\mathrm{vir}}}{\lambda_{-1}(N_{\mathrm{vir},F_0}^{*})}),
\]
where $F_0$ is a component of $\mathbb{C}^*$-fixed locus of the quasimap space $QM^d(T^*G(r,n))$, induced by the natural $\mathbb{C}^*$-action on $\mathbb{P}^1$, and consisting of degree $d$ quasimaps $f$ such that $f\arrowvert_{\mathbb{P}^1 \setminus \{0\}}$ is a constant map to $T^*(G(r,n))$. See also \cite{CKM, MR4167521} for cohomological $I$-function.
\end{definition}
\begin{definition}
    The twisted $I$-function of $T^*G(r,n)$ is defined by:
    \[
   I^{tw}= \sum_{d \geq 0}Q^dev_*(\frac{\hat{\mathcal{O}}^{{\mathrm{vir}}}_{F_0}}{\lambda_{-1}(N_{F_0}^{\mathrm{vir}})}),
    \]
    where we consider $F_0$ equipped with the action of $A=T \times \C_q^* \times \C_\hbar^*$, and replace the virtual structure sheaf in \Cref{I-function definition} with the symmetrized virtual structure sheaf.
\end{definition}

\begin{prop}\label{prop:V=I}
    The twisted $I$-function coincides with the quasimap vertex function.
\end{prop}
\begin{proof}
    For a fixed $d$, we compute $ev_*(\frac{\hat{\mathcal{O}}^{{\mathrm{vir}}}_{F_0}}{\lambda_{-1}(N_{\mathrm{vir},F_0}^{*})})$, using the localization theorem:
    \[
ev_*(\frac{\mathcal{O}_{F_0}^{\mathrm{vir}}}{\lambda_{-1}({N}_{\mathrm{vir},F_0}^{*})}) =\sum_{x \in F_0^{A}}  \frac{\frac{\hat{\mathcal{O}}^{{\mathrm{vir}}}_{F_0}}{\lambda_{-1}({N}_{\mathrm{vir},F_0}^{*})}\arrowvert_{x}}{\lambda_{-1}({T}_{\mathrm{vir},F_0,x}^{*})} = \sum_{x \in F_0^{A}}\frac{\hat{\mathcal{O}}^{{\mathrm{vir}}}_{F_0,x}}{\lambda_{-1}({N}_{\mathrm{vir},F_0,x}^{*}).\lambda_{-1}({T}_{\mathrm{vir},F_0,x}^{*})}
\]
\[
= \sum_{x \in F_0^{A}}\frac{\hat{\mathcal{O}}^{{\mathrm{vir}}}_{F_0,x}}{\lambda_{-1}(T^*_{\mathrm{vir},x}QMap)},
\]
where $F_0^{A}$ denotes $A$-fixed points of $F_0$, which can easily be seen to coincide with those described in \Cref{lemma:qmap-fixed}. Therefore, we can write the last equation as:
\[
\sum_{x \in (QM_{\infty}^d)^{A \times \mathbb{C}_{\hbar}^* \times \mathbb{C}_q^*}}\frac{\hat{\mathcal{O}}^{{\mathrm{vir}}}_{F_0,x}}{\lambda_{-1}(T^*_{\mathrm{vir},x}QMap)}
\]
which is equal to $ev_{{\infty},*}(QM^d_{\infty},\hat{\mathcal{O}}^{{\mathrm{vir}}})$.
\end{proof}

\section{Explicit Example for \texorpdfstring{$\mathbb{P}^1$}{P1}}
In this section, we explicitly compute the vertex function, $V_p(t;q,Q,\hbar)$, of $T^*\mathbb{P}^1$, and the $\lambda_y$ $K$-theoretic balanced $I$-function, $\mathcal{B}_{y}(I(t;q,Q)_{\arrowvert_{p}})$, of $\mathbb{P}^1$ at each fixed point, and we directly verify \cref{main thm}.

Recall that the quiver corresponding to $T^*\mathbb{P}^1$ is given by the following diagram:

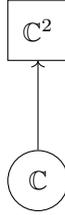
\begin{center}
\begin{tikzpicture}[node distance=1.5cm]
    \node[draw, rectangle, minimum width=0.8cm, minimum height=0.8cm] (A) at (0, 0) {$\mathbb{C}^2$};  
    \node[draw, circle, minimum size=0.8cm] (B) at (0, -2) {$\mathbb{C}$};  

    \draw[->] (B) -- (A) node[midway, right] {};  
\end{tikzpicture}
\end{center}

 The associated moment map is: 
\begin{align*}
   \mu:T^*\mathrm{Hom}(\mathbb{C},\mathbb{C}^2) = \mathrm{Hom}(&\mathbb{C},\mathbb{C}^2) \oplus \mathrm{Hom}(\mathbb{C}^2,\mathbb{C}) \to \mathfrak{gl}(\mathbb{C})=\mathbb{C}  \\& (X,Y) \mapsto YX.
\end{align*} 

Thus:
\[
T^*\mathbb{P}^1 = \big\{(X,Y) \in Hom(\mathbb{C},\mathbb{C}^2) \oplus Hom(\mathbb{C}^2,\mathbb{C}): \text{ $X$ is injective and $YX$=0}\}/{\mathbb{C}^*}\big\}.
\]

In this setting, the torus $T:=(\mathbb{C}^*)^2$ acts on $\mathbb{C}^2$ by $(t_1,t_2).(x,y)=(t_1x,t_2y)$, and  $\mathbb{C}_q^*$  acts on $\mathbb{P}^1$ via $q\cdot[x:y] = [qx:y]$. Additionally, the torus  $\mathbb{C}^*_{\hbar}$, which scales the cotangent direction by the character $\hbar$, acts on the cotangent bundle via $\hbar.(X,Y) = (X,\hbar Y)$.
The $T \times \mathbb{C}^*_{\hbar}$-fixed points of $T^*\mathbb{P}^1$ consist of two points, each corresponding to a one-dimensional subspace of $\mathbb{C}^2$ spanned by a coordinate vector. We denote these fixed points by $p_1$ and $p_2$, where $p_1$ corresponds to the line $\langle e_1\rangle$ and $p_2$ corresponds to $\langle e_2\rangle$.

The degree $d$ stable quasimap space $QM_{\infty}^d( T^*\mathbb{P}^1)$ is defined by triples $\langle \mathcal{V},\mathcal{W},f \rangle$, where:

\begin{itemize}
    \item   $\mathcal{V}$ is a degree $d$ line bundle over $\mathbb{P}^1$;
    \item   $\mathcal{W}$ is a trivial bundle of rank 2 over $\mathbb{P}^1$;
    \item   $f \in H^0(\mathbb{P}^1,\mathcal{M} \oplus \hbar \mathcal{M}^{\vee})$, where $\mathcal{M} = Hom(\mathcal{V}, \mathcal{W})$. 
\end{itemize}
Let $\langle\mathcal{V},\mathcal{W},f\rangle$ be a $T \times \mathbb{C}^*_{\hbar} \times \mathbb{C}^*_q$-fixed stable quasimap of degree $d$. By \cref{lemma:qmap-fixed}, we have:
\begin{equation}\label{fixed quasimap}
    \mathcal{V} = q^d\mathbb{C}_{t_i}\otimes \mathcal{O}_{\mathbb{P}^1}(-d);
\end{equation}
\[
\mathcal{W} = \mathbb{C}_{t_1}\otimes \mathcal{O}_{\mathbb{P}^1} \oplus \mathbb{C}_{t_2}\otimes \mathcal{O}_{\mathbb{P}^1},
\]
for $i=1,2$.

  By \Cref{Virtual tangent bundle}, the virtual tangent bundle at this fixed point is given by: 
 \[
 T_{\mathrm{vir},\langle \mathcal{V},\mathcal{W},f \rangle} = H^{\bullet}(T^{1/2} \oplus \hbar(T^{1/2})^{\vee})- T_{p_i}X,
 \]
where 
\[
T^{1/2} = \mathcal{V}^{\vee} \otimes \mathcal{W} - \mathcal{V}^{\vee} \otimes \mathcal{V},
\]
and $H^{\bullet} = H^0 - H^1$.

Now, we compute $T^{\mathrm{vir}}_{\langle\mathcal{V},\mathcal{W},f\rangle}$ for $i=1$:
\begin{align*}
    T^{1/2} &= q^{-d}\mathbb{C}_{t_{1}^{-1}}\otimes\mathcal{O}_{\mathbb{P}^1}(d) \otimes [\mathbb{C}_{t_1}\otimes\mathcal{O}_{\mathbb{P}^1} \oplus \mathbb{C}_{t_2}\otimes\mathcal{O}_{\mathbb{P}^1}] - \mathcal{O}_{\mathbb{P}^1} \\& = q^{-d}\mathcal{O}_{\mathbb{P}^1}(d) \oplus q^{-d}\mathbb{C}_{\frac{t_2}{t_1}}\otimes\mathcal{O}_{\mathbb{P}^1}(d) - \mathcal{O}_{\mathbb{P}^1}
\end{align*}
Thus, we obtain:
\[
\scalebox{0.95}{$H^{\bullet}(T^{1/2} \oplus \hbar(T^{1/2})^{\vee}) = H^{\bullet}(q^{-d}\mathcal{O}(d) \oplus \hbar q^d \mathcal{O}(-d)) \oplus H^{\bullet}(q^{-d} \mathbb{C}_{\frac{t_2}{t_1}}\mathcal{O}(d) \oplus \hbar q^d \mathbb{C}_{\frac{t_1}{t_2}} \mathcal{O}(-d)) - H^{\bullet}(\mathcal{O} - \hbar\mathcal{O})$}.
\]
Note that the character of $H^0(\mathcal{O}(d))$ is:
\[
1+q+\dots+q^d.
\]
After computing the cohomology groups, the weight space decomposition of the virtual tangent bundle at $\langle\mathcal{V}, \mathcal{W},f\rangle$ is:
\begin{equation}\label{vertexequation}
T^{\mathrm{vir}}_{\langle\mathcal{V}, \mathcal{W},f\rangle} =(q^{-1}+\dots+q^{-d})+(\frac{t_2}{t_1}q^{-1}+\dots+\frac{t_2}{t_1}q^{-d})  \\ -(\hbar+\hbar q+\dots+\hbar q^{d-1})-(\hbar\frac{t_1}{t_2}+\hbar\frac{t_1}{t_2}q+\dots+\hbar\frac{t_1}{t_2}q^{d-1}).
\end{equation}
Recall from \Cref{vertex def} that the Okounkov's vertex function is: 
\[
V(Q) =  \sum_d Q^d ev{_{\infty}}_*(\hat{\mathcal{O}}^{\mathrm{vir}}_d),
\]

where $\hat{\mathcal{O}}_{\mathrm{vir}}^d = \mathcal{O}_{\mathrm{vir}}^d \otimes \mathcal{K}_{\mathrm{vir}}^{\frac{1}{2}}$ and $\mathcal{K}_{\mathrm{vir}} = det^{-1}T_{\mathrm{vir}}QM_{\infty}^d$.\\

By \Cref{lastsection}, we have:
\begin{align*}
    V_{p_1}(t;q,Q,\hbar) = \sum_{d=0}^{\infty}V_{p_1}^d(t;q,\hbar)Q^d= \sum_{d=0}^{\infty} \hat{s}(T_{{\mathrm{vir},\langle \mathcal{V},\mathcal{W},f \rangle}_d})Q^{d},
\end{align*}
where for each $d$, $\langle \mathcal{V},\mathcal{W},f \rangle_d$ denotes the unique stable quasimap of degree $d$ satisfying the condition $f(\infty) = p_1$, as described in \Cref{fixed quasimap}.

Now, by \Cref{vertexequation}, we can compute $\hat{s}(T_{{\mathrm{vir},\langle \mathcal{V},\mathcal{W},f \rangle}_d})$:
\begin{align*}
    &\scalebox{1}{$\hat{s}(q^{-1}+\dots+q^{-d}-(\hbar+\hbar q+\dots+\hbar q^{d-1})+(\frac{t_2}{t_1}q^{-1}+\dots+ \frac{t_2}{t_1}q^{-d})-(\hbar \frac{t_1}{t_2}+\hbar \frac{t_1}{t_2}q+\dots+\hbar \frac{t_1}{t_2}q^{d-1}))$} \\& =\frac{\hat{s}(q^{-1}+\dots+q^{-d})}{\hat{s}(\hbar+\hbar q+\dots+\hbar q^{d-1})}\cdot\frac{\hat{s}(\frac{t_2}{t_1}q^{-1}+\dots+\frac{t_2}{t_1}q^{-d})}{\hat{s}(\hbar \frac{t_1}{t_2}+\hbar \frac{t_1}{t_2}q+\dots+\hbar \frac{t_1}{t_2}q^{d-1}))}\\& =\frac{\hat{s}(q^{-1})\dots\hat{s}(q^{-d}))}{\hat{s}(\hbar)\dots\hat{s}(\hbar q^{d-1})}\cdot\frac{\hat{s}(\frac{t_2}{t_1}q^{-1})\dots\hat{s}(\frac{t_2}{t_1}q^{-d})}{\hat{s}(\hbar \frac{t_1}{t_2})\dots\hat{s}(\hbar \frac{t_1}{t_2}q^{d-1})} = \frac{(\hbar )_d}{(q)_d}\cdot\frac{(\hbar \frac{t_1}{t_2})_d}{(q\frac{t_1}{t_2})_d}\cdot(\frac{q^2}{\hbar})^d.
\end{align*}
Thus, we obtain:
\begin{align*}
    V_{p_1}(t;q,Q,\hbar) = \sum_{d=0}^{\infty}\frac{(\hbar)_d}{(q)_d}\cdot\frac{(\hbar\frac{t_1}{t_2})_d}{(q\frac{t_1}{t_2})_d}\cdot(\frac{q^2Q}{\hbar})^d.
\end{align*}

By changing variables $\frac{q^2Q}{\hbar} \to Q$, we get:
\begin{align*}
   V_{p_1}(t;q,Q,\hbar) = \sum_{d=0}^{\infty}\frac{(\hbar)_d}{(q)_d}.\frac{(\hbar\frac{t_1}{t_2})_d}{(q\frac{t_1}{t_2})_d}Q^d 
\end{align*}
Similarly,
\[
V_{p_2}(t;q,Q,\hbar) = \sum_{d=0}^{\infty}\frac{(\hbar)_d}{(q)_d}.\frac{(\hbar \frac{t_2}{t_1})_d}{(q\frac{t_2}{t_1})_d}Q^d
\]

\begin{remark}
We can also express the formula without using the Pochhammer symbol. Recall that the Pochhammer symbol is defined by: 
\[
(x)_d = (1-x)(1-xq)\dots(1-xq^{d-1}),
\]
and thus we can rewrite $V_{p_1}(z)$ as follows:
\begin{equation}\label{pochformula}
   V_{p_1}(t;q,Q,\hbar) = \sum_{d=0}^{\infty}\frac{(1-\hbar)(1-\hbar q)\dots(1-\hbar q^{d-1})}{(1-q)(1-q^2)\dots(1-q^d)}\cdot\frac{(1-\hbar\frac{t_1}{t_2})(1-\hbar\frac{t_1}{t_2}q)\dots(1-\hbar\frac{t_1}{t_2}q^{d-1})}{(1-\frac{t_1}{t_2}q)(1-\frac{t_1}{t_2}q^2)\dots(1-\frac{t_1}{t_2}q^d)} Q^d.
\end{equation}
\end{remark}

Now we compute $I(t;q,Q)_{\arrowvert_{p_i}} = \sum_{d \geq 0} I_d(t;q)_{\arrowvert_{p_i}}Q^d$. To do this, we localize $I_d(t;q)$ at the fixed points $p_i$, corresponding to the one-dimensional subspaces $<e_i>$. For simplicity assume $i=1$. 
The unique fixed point of degree $d$, supported at zero in the Quot scheme corresponding to $p_1$, is described by the injection: 
\[
0 \to \mathbb{C}_{t_1}\otimes \mathcal{O}_{\mathbb{P}^1}(-d) \to \mathbb{C}_{t_1}\otimes \mathcal{O}_{\mathbb{P}^1} \oplus \mathbb{C}_{t_2}\otimes\mathcal{O}_{\mathbb{P}^1}.
\]
This injection is given by the natural map  $\mathbb{C}_{t_1}\otimes \mathcal{O}_{\mathbb{P}^1}(-d) \to \mathbb{C}_{t_1}\otimes \mathcal{O}_{\mathbb{P}^1}$, defined by multiplication by $x^d$. 
We denote this fixed point by $a$. We then obtain the following expression for the tangent space of the Quot scheme at $a$:
\[
    T_aQuot_{\mathbb{P}^1,d}(\mathbb{C},1) = Hom(\mathbb{C}_{t_1}\otimes \mathcal{O}_{\mathbb{P}^1}(-d),\frac{\mathbb{C}_{t_1}\otimes \mathcal{O}_{\mathbb{P}^1} \oplus \mathbb{C}_{t_2}\otimes\mathcal{O}_{\mathbb{P}^1}}{\mathbb{C}_{t_1}\otimes\mathcal{O}_{\mathbb{P}^1}(-d)}).
\]
By applying \Cref{fixed-point lemma}, we obtain that the weight space decomposition for the dual of $Hom(\mathbb{C}_{t_1}\otimes \mathcal{O}_{\mathbb{P}^1}(-d),\frac{t_1\mathcal{O}_{\mathbb{P}^1}}{\mathbb{C}_{t_1}\otimes \mathcal{O}_{\mathbb{P}^1}(-d)})$ is:
\[
q+q^2+\dots+q^d.
\]
Similarly, the weight decomposition for the dual of $Hom(\mathbb{C}_{t_1}\otimes \mathcal{O}(-d),\mathbb{C}_{t_2}\otimes \mathcal{O})$ is:
\[
    \frac{t_1}{t_2}+\frac{t_1}{t_2}q+\dots+\frac{t_1}{t_2}q^d.
\]

Thus:
\[
I_d(t;q)_{\arrowvert_{p_1}}=\frac{1}{(1-q)(1-q^2)\dots(1-q^d)}\cdot\frac{1}{(1-\frac{t_1}{t_2}q)\dots(1-\frac{t_1}{t_2}q^d)}.
\]
therefore, we obtain:
\begin{align*}
 \mathcal{B}_{y}(I_d(t;q)_{\arrowvert_{p_1}})= \frac{(1+yq)(1+yq^2)\dots(1+yq^d)}{(1-q)(1-q^2)\dots(1-q^d)}.\frac{(1+y\frac{t_1}{t_2}q)\dots(1+y\frac{t_1}{t_2}q^d)}{(1-\frac{t_1}{t_2}q)\dots(1-\frac{t_1}{t_2}q^d)}.
\end{align*}
substituting $y=-q^{-1}\hbar$, we get:
\begin{align*}
\mathcal{B}_{-q^{-1}\hbar}(I_d(t;q)_{\arrowvert_{p_1}}) = \frac{(1-\hbar)(1-\hbar q)\dots(1-\hbar q^{d-1})}{(1-q)(1-q^2)\dots(1-q^d)}\cdot\frac{(1-\hbar \frac{t_1}{t_2})\dots(1-\hbar \frac{t_1}{t_2}q^{d-1})}{(1-\frac{t_1}{t_2}q)\dots(1-\frac{t_1}{t_2}q^d)}.
\end{align*}

This expression matches exactly $V^d_p(t;q,\hbar)$, the coefficient of $Q^d$ in the expansion of   \Cref{pochformula}. Similarly, we see that $\mathcal{B}_{-q^{-1}\hbar}(I_d(t;q,Q)_{\arrowvert_{p_2}}) = V_{p_2}(t;q,Q,\hbar)$.
\appendix
\section{}\label{Appendix A}
Let $(d_1,\dots,d_r) \in \mathbb{Z}_{\geq 0}^r$, and let $d = \max\{d_1,\dots,d_r\}$. Define the set $W_{(d_1,\dots,d_r)} \subset K_{T \times \mathbb{C}^*_{\lambda} \times \mathbb{C}^*_q}(\mathrm{pt})$ as follows:
\[
W_{(d_1,\dots,d_r)} = \big\{ q^k \frac{t_i}{t_j}\lambda^l : -d \leq k \leq d, 1 \leq i,j \leq n, l \in \{0,1\}\big\}.
\]
Set $W = \bigcup_{(d_1,\dots,d_r)}W_{(d_1,\dots,d_r)}$. For each element $q^k \frac{t_i}{t_j}\lambda^l \in W$, define a one-dimensional $T \times \mathbb{C}^*_{\lambda} \times \mathbb{C}^*_q$-module with that weight. We define the set $U$ by:
\[
U := \big\{\frac{1}{\prod_{i}\lambda_{-1}V_i}: \text{$V_i$ is a one-dimensional $T \times \mathbb{C}^*_{\lambda} \times \mathbb{C}^*_q$-module, whose weight is nontrivial element of $W$ } \big\}
\]
We also extend the definition of $\mathcal{B}_y$ to quotients of elements of $U$, i.e., if $\frac{a}{b}$ with $a,b \in U$, then we define:
\[
\mathcal{B}_y(\frac{a}{b}):= \frac{\mathcal{B}_y(a)}{\mathcal{B}_y(b)}.
\]

With the above definitions, let $\mathcal{J}$ denote the $\lambda_y$-balanced $K$-theoretic class of $\overline{J}^{tw}$ given in \eqref{twisted function}. That is, $\mathcal{J}:= \mathcal{B}_y(\overline{J}^{tw})$, and we obtain:
\[
\mathcal{J} = \sum_{d_1,\dots,d_n\geq 0}\prod_{i=1}^rQ_i^{d_i}\frac{\prod_{i=1}^n\prod_{j=1}^{n}\prod_{m=1}^{d_i}(1+yq^m\frac{P_i}{t_j})}{\prod_{i=1}^n\prod_{j=1}^{n}\prod_{m=1}^{d_i}(1-q^m\frac{P_i}{t_j})}\prod_{i\neq j}\frac{\prod_{m=-\infty}^{d_i-d_j} (1-q^m\lambda\frac{P_i}{P_j})\prod_{m=-\infty}^{0} (1+yq^m\lambda\frac{P_i}{P_j})}{\prod_{m=-\infty}^{0} (1-q^m\lambda\frac{P_i}{P_j}) \prod_{m=-\infty}^{d_i-d_j} (1+yq^m\lambda\frac{P_i}{P_j})}.
\]

\section{}\label{Appendix B}
We prove the $q$-difference equations \eqref{q-difference eq} from \Cref{Bethe section}. For simplicity of notation, we consider only the case $i=1$. 

Let:
\[
\mathcal{J}=\sum_{d_1,\dots,d_n\geq 0}\prod_{i=1}^rQ_i^{d_i}\frac{\prod_{i=1}^n\prod_{j=1}^{n}\prod_{m=1}^{d_i}(1+yq^m\frac{P_i}{t_j})}{\prod_{i=1}^n\prod_{j=1}^{n}\prod_{m=1}^{d_i}(1-q^m\frac{P_i}{t_j})}\prod_{i\neq j}\frac{\prod_{m=-\infty}^{d_i-d_j} (1-q^m\lambda\frac{P_i}{P_j})\prod_{m=-\infty}^{0} (1+yq^m\lambda\frac{P_i}{P_j})}{\prod_{m=-\infty}^{0} (1-q^m\lambda\frac{P_i}{P_j}) \prod_{m=-\infty}^{d_i-d_j} (1+yq^m\lambda\frac{P_i}{P_j})}.
\]
\[
\mathcal{D}_1^1 = \prod_{j\neq 1}(1+y\lambda P_1P_j^{-1}q^{Q_1\partial_{Q_1}-Q_j\partial_{Q_j}})\prod_{j\neq 1} (1-\lambda qP_jP_1^{-1}q^{Q_j\partial_{Q_j}-Q_1\partial_{Q_1}})\prod_a(1-\frac{P_1}{t_a}q^{Q_1\partial_{Q_1}}).
\]
\[
\mathcal{D}_2^1 = \prod_{j\neq 1} (1+y\lambda qP_jP_1^{-1}q^{Q_j\partial_{Q_j}-Q_1\partial_{Q_1}})\prod_a(1+y\frac{P_1}{t_a}q^{Q_1\partial_{Q_1}})Q_1\prod_{j\neq 1}(1-\lambda qP_1P_j^{-1}q^{Q_1\partial_{Q_1}-Q_j\partial_{Q_j}}).
\]
 We have the following result.
\begin{thm}
  $\mathcal{D}_1^1\mathcal{J} = \mathcal{D}_2^1\mathcal{J}$.  
\end{thm}
\begin{proof}
Fix the following notation:
\[
E:= \frac{\prod_{i=2}^n\prod_{j=1}^{n}\prod_{m=1}^{d_i}(1+yq^m\frac{P_i}{t_j})}{\prod_{i=2}^n\prod_{j=1}^{n}\prod_{m=1}^{d_i}(1-q^m\frac{P_i}{t_j})}, \quad  F:= \prod_{i\neq j}^{i,j \geq2}\frac{\prod_{m=-\infty}^{d_i-d_j} (1-q^m\lambda\frac{P_i}{P_j})\prod_{m=-\infty}^{0} (1+yq^m\lambda\frac{P_i}{P_j})}{\prod_{m=-\infty}^{0} (1-q^m\lambda\frac{P_i}{P_j}) \prod_{m=-\infty}^{d_i-d_j} (1+yq^m\lambda\frac{P_i}{P_j})}.
\]
First, we find the coefficient of $Q_1^{d_1+1}Q_2^{d_2} \dots Q_r^{d_r}$ in $\mathcal{D}_1^1\mathcal{J}$. 

We proceed in three steps.

$\bullet {\bf Step1}$:

    Find the coefficient in $\mathcal{J}_1 := \prod_a(1-\frac{P_1}{t_a}q^{Q_1\partial_{Q_1}})\mathcal{J}$ and denote it by $A_1$. We obtain:
    \[
    \scalebox{0.95}{$A_1=E\frac{\prod_{j=1}^{n}\prod_{m=1}^{d_1+1}(1+yq^m\frac{P_1}{t_j})}{\prod_{j=1}^{n}\prod_{m=1}^{d_1}(1-q^m\frac{P_1}{t_j})}\prod_{j\neq 1}\frac{\prod_{m=-\infty}^{d_1+1-d_j} (1-q^m\lambda\frac{P_1}{P_j})\prod_{m=-\infty}^{0} (1+yq^m\lambda\frac{P_1}{P_j})}{\prod_{m=-\infty}^{0} (1-q^m\lambda\frac{P_1}{P_j}) \prod_{m=-\infty}^{d_1+1-d_j} (1+yq^m\lambda\frac{P_1}{P_j})} \cdot \frac{\prod_{m=-\infty}^{d_j-d_1-1} (1-q^m\lambda\frac{P_j}{P_1})\prod_{m=-\infty}^{0} (1+yq^m\lambda\frac{P_j}{P_1})}{\prod_{m=-\infty}^{0} (1-q^m\lambda\frac{P_j}{P_1}) \prod_{m=-\infty}^{d_j-d_1-1} (1+yq^m\lambda\frac{P_j}{P_1})}F$.}
    \]
    
$\bullet {\bf Step 2}$:

    Find the coefficient in $ \mathcal{J}_2 := \prod_{j\neq 1} (1-\lambda qP_jP_1^{-1}q^{Q_j\partial_{Q_j}-Q_1\partial_{Q_1}})\mathcal{J}_1$. Denote it by $A_2$, then:
    \[
    \scalebox{0.95}{$A_2=E\frac{\prod_{j=1}^{n}\prod_{m=1}^{d_1+1}(1+yq^m\frac{P_1}{t_j})}{\prod_{j=1}^{n}\prod_{m=1}^{d_1}(1-q^m\frac{P_1}{t_j})}\prod_{j\neq 1}\frac{\prod_{m=-\infty}^{d_1+1-d_j} (1-q^m\lambda\frac{P_1}{P_j})\prod_{m=-\infty}^{0} (1+yq^m\lambda\frac{P_1}{P_j})}{\prod_{m=-\infty}^{0} (1-q^m\lambda\frac{P_1}{P_j}) \prod_{m=-\infty}^{d_1+1-d_j}(1+yq^m\lambda\frac{P_1}{P_j})} \cdot \frac{\prod_{m=-\infty}^{d_j-d_1} (1-q^m\lambda\frac{P_j}{P_1})\prod_{m=-\infty}^{0} (1+yq^m\lambda\frac{P_j}{P_1})}{\prod_{m=-\infty}^{0} (1-q^m\lambda\frac{P_j}{P_1}) \prod_{m=-\infty}^{d_j-d_1-1} (1+yq^m\lambda\frac{P_j}{P_1})}F$.}
    \]

    $\bullet {\bf Sterp 3}$:

    Find the coefficient in $\mathcal{D}_1^1\mathcal{J} = \prod_{j\neq 1}(1+y\lambda P_1P_j^{-1}q^{Q_1\partial_{Q_1}-Q_j\partial_{Q_j}})\mathcal{J}_2$. Denote it by $A_3$, and we obtain:
    \[
    \scalebox{0.95}{$A_3=E\frac{\prod_{j=1}^{n}\prod_{m=1}^{d_1+1}(1+yq^m\frac{P_1}{t_j})}{\prod_{j=1}^{n}\prod_{m=1}^{d_1}(1-q^m\frac{P_1}{t_j})}\prod_{j\neq 1}\frac{\prod_{m=-\infty}^{d_1+1-d_j} (1-q^m\lambda\frac{P_1}{P_j})\prod_{m=-\infty}^{0} (1+yq^m\lambda\frac{P_1}{P_j})}{\prod_{m=-\infty}^{0} (1-q^m\lambda\frac{P_1}{P_j}) \prod_{m=-\infty}^{d_1-d_j}(1+yq^m\lambda\frac{P_1}{P_j})} \cdot \frac{\prod_{m=-\infty}^{d_j-d_1} (1-q^m\lambda\frac{P_j}{P_1})\prod_{m=-\infty}^{0} (1+yq^m\lambda\frac{P_j}{P_1})}{\prod_{m=-\infty}^{0} (1-q^m\lambda\frac{P_j}{P_1}) \prod_{m=-\infty}^{d_j-d_1-1} (1+yq^m\lambda\frac{P_j}{P_1})}F$.}
    \]
    This is the coefficient of $Q_1^{d_1+1} \dots Q_r^{d_r}$ in $\mathcal{D}_1^1\mathcal{J}$. 
    
    Now we compute the coefficient of $Q_1^{d_1+1}Q_2^{d_2} \dots Q_r^{d_r}$ in $\mathcal{D}_2^1\mathcal{J}$, again in three steps.

    $\bullet {\bf Step1}$:
    
    Find the coefficient of $Q_1^{d_1}Q_2^{d_2} \dots Q_r^{d_r}$ in $\mathcal{I}_1 := Q_1\prod_{j\neq 1}(1-\lambda qP_1P_j^{-1}q^{Q_1\partial_{Q_1}-Q_j\partial_{Q_j}}) \mathcal{J}$, denoted by $B_1$:
    \[
    \scalebox{0.95}{$B_1=E\frac{\prod_{j=1}^{n}\prod_{m=1}^{d_1}(1+yq^m\frac{P_1}{t_j})}{\prod_{j=1}^{n}\prod_{m=1}^{d_1}(1-q^m\frac{P_1}{t_j})}\prod_{j\neq 1}\frac{\prod_{m=-\infty}^{d_1+1-d_j} (1-q^m\lambda\frac{P_1}{P_j})\prod_{m=-\infty}^{0} (1+yq^m\lambda\frac{P_1}{P_j})}{\prod_{m=-\infty}^{0} (1-q^m\lambda\frac{P_1}{P_j}) \prod_{m=-\infty}^{d_1-d_j} (1+yq^m\lambda\frac{P_1}{P_j})} \cdot \frac{\prod_{m=-\infty}^{d_j-d_1} (1-q^m\lambda\frac{P_j}{P_1})\prod_{m=-\infty}^{0} (1+yq^m\lambda\frac{P_j}{P_1})}{\prod_{m=-\infty}^{0} (1-q^m\lambda\frac{P_j}{P_1}) \prod_{m=-\infty}^{d_j-d_1} (1+yq^m\lambda\frac{P_j}{P_1})}F$.}
    \]
    Note that in this step the degree of $Q_1$ increases to $d_1+1$. Therefore, in the next step, we find the coefficient of $Q_1^{d_1+1} \dots Q_r^{d_r}$, denoted by $B_2$.

    $\bullet {\bf Step 2}$:

    In $\mathcal{I}_2 := \prod_a(1+y\frac{P_1}{t_a}q^{Q_1\partial_{Q_1}})\mathcal{I}_1$, we have:
    \[
    \scalebox{0.95}{$B_2=E\frac{\prod_{j=1}^{n}\prod_{m=1}^{d_1+1}(1+yq^m\frac{P_1}{t_j})}{\prod_{j=1}^{n}\prod_{m=1}^{d_1}(1-q^m\frac{P_1}{t_j})}\prod_{j\neq 1}\frac{\prod_{m=-\infty}^{d_1+1-d_j} (1-q^m\lambda\frac{P_1}{P_j})\prod_{m=-\infty}^{0} (1+yq^m\lambda\frac{P_1}{P_j})}{\prod_{m=-\infty}^{0} (1-q^m\lambda\frac{P_1}{P_j}) \prod_{m=-\infty}^{d_1-d_j} (1+yq^m\lambda\frac{P_1}{P_j})} \cdot \frac{\prod_{m=-\infty}^{d_j-d_1} (1-q^m\lambda\frac{P_j}{P_1})\prod_{m=-\infty}^{0} (1+yq^m\lambda\frac{P_j}{P_1})}{\prod_{m=-\infty}^{0} (1-q^m\lambda\frac{P_j}{P_1}) \prod_{m=-\infty}^{d_j-d_1} (1+yq^m\lambda\frac{P_j}{P_1})}F$.}
    \]

    $\bullet {\bf Step 3}$:
    
    Finally, compute the coefficient in $\mathcal{D}_2^1\mathcal{J} = \prod_{j\neq 1} (1+y\lambda qP_jP_1^{-1}q^{Q_j\partial_{Q_j}-Q_1\partial_{Q_1}}) \mathcal{I}_2$, denoted by $B_3$:
    \[
     \scalebox{0.93}{$B_3=E\frac{\prod_{j=1}^{n}\prod_{m=1}^{d_1+1}(1+yq^m\frac{P_1}{t_j})}{\prod_{j=1}^{n}\prod_{m=1}^{d_1}(1-q^m\frac{P_1}{t_j})}\prod_{j\neq 1}\frac{\prod_{m=-\infty}^{d_1+1-d_j} (1-q^m\lambda\frac{P_1}{P_j})\prod_{m=-\infty}^{0} (1+yq^m\lambda\frac{P_1}{P_j})}{\prod_{m=-\infty}^{0} (1-q^m\lambda\frac{P_1}{P_j}) \prod_{m=-\infty}^{d_1-d_j} (1+yq^m\lambda\frac{P_1}{P_j})} \cdot \frac{\prod_{m=-\infty}^{d_j-d_1} (1-q^m\lambda\frac{P_j}{P_1})\prod_{m=-\infty}^{0} (1+yq^m\lambda\frac{P_j}{P_1})}{\prod_{m=-\infty}^{0} (1-q^m\lambda\frac{P_j}{P_1}) \prod_{m=-\infty}^{d_j-d_1-1} (1+yq^m\lambda\frac{P_j}{P_1})}F$.}
    \]
    Since $A_3 = B_3$, and $A_3$ is the coefficient of $Q_1^{d_1+1} \dots Q_r^{d_r}$ in $\mathcal{D}_1^1$ while $B_3$ is the coefficient of $Q_1^{d_1+1} \dots Q_r^{d_r}$ in $\mathcal{D}_1^2$, we conclude that $\mathcal{D}_1^1\mathcal{J} = \mathcal{D}_1^2\mathcal{J}$, as desired.
\end{proof}

\bibliographystyle{alpha}
\bibliography{biblio}

\end{document}